\documentclass[11pt,reqno]{amsproc}
\pdfoutput=1 

\usepackage{amsmath,amsfonts,amssymb,amsthm}
\usepackage{mathtools}
\usepackage{mathrsfs}
\usepackage{enumerate}
\usepackage{graphics,graphicx}
\usepackage{color}
\usepackage{times}

\usepackage{subcaption}
\usepackage{tikz}
\usetikzlibrary{decorations.pathmorphing}
\usepackage{pgfplots}
\usepgfplotslibrary{groupplots}
\pgfplotsset{compat=1.10}
\pgfplotsset{
    tick align=outside,
    x grid style={white},
    xmajorgrids,
    y grid style={white},
    ymajorgrids,
    axis line style={white},
    axis background/.style={fill=white!92!black},
    legend style={draw=white, fill=white},
    legend cell align={left}
}

\usepackage[colorlinks=true, pdfstartview=FitV, linkcolor=blue, citecolor=blue, urlcolor=blue]{hyperref}
\usepackage[margin=1in]{geometry}
\usepackage[capitalise]{cleveref}

\newtheorem{theorem}{Theorem}
\newtheorem{proposition}{Proposition}[section]
\newtheorem{lemma}[proposition]{Lemma}
\newtheorem{corollary}[proposition]{Corollary}

\theoremstyle{definition}
\newtheorem{definition}[proposition]{Definition}

\newtheorem{remark}[proposition]{Remark}

\numberwithin{equation}{section}

\newcommand{\CC}{\mathbb{C}}
\newcommand{\NN}{\mathbb{N}}
\newcommand{\ZZ}{\mathbb{Z}}
\newcommand\TT {{\mathbb T}}
\newcommand\RR {{\mathbb R}}
\newcommand{\bbS}{\mathbb{S}}
\newcommand{\bbP}{\mathbb{P}}

\newcommand\eps{\varepsilon}
\newcommand\e{{\mathrm{e}}}
\newcommand\dd{{\mathrm{d}}}
\newcommand\ddt{{\frac{\dd}{\dd t}}}
\def\l {\langle}
\def\r {\rangle}
\def\Re{\operatorname{Re}}
\def\Im{\operatorname{Im}}
\newcommand{\Id}{\mathrm{I}} 
\DeclareMathOperator{\tr}{tr} 
\DeclareMathOperator{\ricci}{Ric}
\DeclareMathOperator{\supp}{supp} 
\newcommand{\ind}{{1\!\!1}} 
\newcommand{\conj}[1]{\overline{#1}} 

\newcommand{\tnorm}[1]{{\left\vert\kern-0.25ex\left\vert\kern-0.25ex\left\vert #1
    \right\vert\kern-0.25ex\right\vert\kern-0.25ex\right\vert}}

\newcommand\cL{{\mathcal L}}

\begin{document}

\title{Orientation mixing in active suspensions}
\author[M. Coti Zelati]{Michele Coti Zelati}
\address{Department of Mathematics, Imperial College London, London, SW7 2AZ, UK}
\email{m.coti-zelati@imperial.ac.uk}
\author[H. Dietert]{Helge Dietert}
\address{Universit\'e Paris Cité and Sorbonne Universit\'e, CNRS,
  Institut de Math\'ematiques de Jussieu-Paris Rive Gauche (IMJ-PRG),
  F-75013, Paris, France}
\email{helge.dietert@imj-prg.fr}
\author[D. Gérard-Varet]{David Gérard-Varet}
\address{Universit\'e Paris Cité and Sorbonne Universit\'e, CNRS,
  Institut de Math\'ematiques de Jussieu-Paris Rive Gauche (IMJ-PRG),
  F-75013, Paris, France}
\email{david.gerard-varet@u-paris.fr}

\begin{abstract}
  We study a popular kinetic model introduced by Saintillan and
  Shelley for the dynamics of suspensions of active elongated
  particles where the particles are described by a distribution in
  space and orientation. The uniform distribution of particles is the
  stationary state of incoherence which is known to exhibit a phase
  transition. We perform an extensive study of the linearised
  evolution around the incoherent state. We show (i) in the
  non-diffusive regime corresponding to spectral (neutral) stability
  that the suspensions experience a mixing phenomenon similar to
  Landau damping and we provide optimal pointwise in time decay rates
  in weak topology. Further, we show (ii) in the case of small
  rotational diffusion \(\nu\) that the mixing estimates persist up to
  time scale \(\nu^{-1/2}\) until the exponential decay at enhanced
  dissipation rate \(\nu^{1/2}\) takes over.

  The interesting feature is that the usual velocity variable of
  kinetic models is replaced by an orientation variable on the
  sphere. The associated \emph{orientation mixing} leads to limited
  algebraic decay for macroscopic quantities. For the proof, we start
  with a general pointwise decay result for Volterra equations that
  may be of independent interest. While, in the non-diffusive case,
  explicit formulas on the sphere allow to conclude the desired decay,
  much more work is required in the diffusive case: here we prove
  mixing estimates for the advection-diffusion equation on the sphere
  by combining an optimized hypocoercive approach with the vector field
  method. One main point in this context is to identify good commuting
  vector fields for the advection-diffusion operator on the sphere.
  Our results in this direction may be useful to other models in
  collective dynamics, where an orientation variable is involved.
\end{abstract}

\maketitle

\tableofcontents

\section{Active suspensions}

Over the last decade, a significant mathematical effort has been put into
the understanding of mixing mechanisms in kinetic equations. The
easiest example is free transport
\begin{equation}\label{eq:euclidean-free-transport}
  \partial_t f + v \cdot \nabla_x f = 0, \quad
  f(0,x,v) = f^{in}(x,v), \quad
  x \in \TT^d, v \in \RR^d.
\end{equation}
The explicit representation $f(t,x,v) = f^{in}(x-vt,v)$ shows a
filamentation of the support of the solution through time, leading to
convergence of the solution as time goes to infinity, in weak
topology, despite the absence of diffusive mechanisms. The rate of
convergence depends on the regularity of the data, from exponential
convergence for analytic data, to polynomial convergence for Sobolev
data. When diffusion is added, leading to
$$   \partial_t f + v \cdot \nabla_x f - \nu \Delta_v f = 0, $$
the small scales created by the filamentation allow for an
acceleration of the diffusive process; this leads to an enhanced
dissipation on time scales \(\nu^{-1/3}\) shorter than the usual
$\nu^{-1}$.

Identifying these phenomena in more complex kinetic equations, either
of transport or weakly dissipative type, linearly and nonlinearly, has
attracted a lot of attention. A main example is the analysis of Landau
damping in the Vlasov-Poisson equation of plasma physics, with
pioneering works \cite{landau-1946,mouhot-villani-2011-landau}. More
recent works were dedicated to weakly dissipative cases, trying to
exhibit both Landau damping and enhanced diffusion at the same time
\cite{Bedrossian17,CLN21}.

Another field where similar mixing phenomena are identified is the one
of collective dynamics. A famous example in this direction is the
Kuramoto model, a system of ODEs describing the dynamics of coupled
oscillators, known to exhibit phase transition, from incoherence to
synchronization. At the mathematical level, the mean-field limit of
this system of ODEs leads to a kinetic equation, with unknown
$f(t,\theta,\omega)$ describing the fraction of oscillators with phase
$\theta \in \TT$ and natural frequency $\omega \in \RR$. Depending on
a bifurcation parameter quantifying the intensity of the coupling,
solutions converge weakly either to the uniform distribution
(incoherence), or to the so-called partially locked states modeling
synchronization (and containing Dirac masses in $\theta$). It turns
out that this convergence is again due to a mixing process, as noticed
for the first time in
\cite{strogatz-mirollo-matthews-1992-coupled}. Such phenomenon
is now fully confirmed mathematically, both linearly and nonlinearly
\cite{chiba-2015-proof-kuramoto-conjecture,dietert-2016-stability-kuramoto,fernandez-gerard-varet-giacomin-2016-landau-kuramoto,benedetto-caglioti-montemagno-2016-exponential-dephasing,dietert-fernandez-gerard-varet-2018-landau-kuramoto,dietert-fernandez-2018-kuramoto}.

In this paper, we will study mixing properties of a popular model describing a dilute suspension of elongated active particles in a Stokes flow.

\subsection{The model}
In \cite{saintillan-shelley-2008-instabilities}, see also  \cite{saintillan-2018-rheology} for a review, D. Saintillan and M. Shelley have introduced a system of PDEs to describe
the dynamics of a dilute suspension of elongated active particles in a
viscous Stokes flow. The word ``active'' refers to the fact that they
convert chemical energy into mechanical work. A typical example are
bacteria, which are able to swim and create stress through the use of
their flagellas. To write down the model, the first step is to
consider a collection of $N$ elongated particles (ellipsoids) immersed
in a Stokes flow. They are described by the position of their center
of mass $x_i \in \TT_L^3 = (\RR/(L \ZZ))^3$ and their director
$p_i \in \bbS^2$, $1 \le i \le N$. The choice of a torus of size $L$
for the spacial domain instead of a real container is for mathematical
convenience: it introduces a typical length scale $L$ without problems
related to the boundary of the domain.  Neglecting interaction between
the particles, one can use the work of Jeffery about a single passive
ellipsoidal particle in a Stokes flow, see
\cite{jeffery-1922-ellipsoidal,taylor-1923}. If the typical size of
the particle is very small compared to $L$, the torque experienced by
the particle $i$ is approximated by
$\mathbb{P}_{p_i^\perp} \left( \gamma E(u) + W(u) \right)\vert_{x_i}
p_i$, where $\mathbb{P}_{p_i^\perp} = {\rm I} - p_i \otimes p_i$ is
the projection orthogonal to $p_i$, while $E(u)$ and $W(u)$ are the
symmetric and skew-symmetric parts of $\nabla u$. The parameter
$\gamma$ is related to the aspect ratio of the ellipsoid, with $\gamma = 1$ in the limit case of a rod.
Note that in the work of Jeffery, $u$ refers to the unperturbed
velocity field, that is the one forgetting the particles. It is thus
defined everywhere on $\TT_L^3$, especially in $x_i$.  The dynamics is
then given by
$$  \dot{x}_i = u(x_i) + U_0 p_i, \quad  \dot{p}_i = \mathbb{P}_{p_i^\perp} \big( \gamma E(u) + W(u)\big)\vert_{x_i} p_i + \text{possible Brownian noise}. $$
The first equation describes the velocity of the particles as a sum of
two contributions: the first one corresponds to advection by the
velocity field $u$ of the flow. The second one corresponds to swimming
along the director $p_i$, at constant speed $U_0 > 0$. The second
equation describes the rotation of the particle, which is known in the
absence of noise to experience periodic trajectories (Jeffery's
orbits).  As far as we understand, these orbits are not really
observed in experiments, and it is more accurate to perturb them by
adding additional rotational Brownian motion, of small amplitude
$\nu \ll 1$. On the other hand, we neglect Brownian motion in $x$,
that is translational diffusion. As pointed out in
\cite[Page~7]{saintillan-shelley-2008-instabilities}, it is not
expected to play a big role in the linear stability properties of the
model. We stress though that it could be included in the analysis
below, and all the results would then hold independently of the
strength of this diffusion.

Besides the swimming velocity $U_0 p_i$, another feature of active
suspensions that needs to be retained is the stress that they create
on the fluid. Saintillan and Shelley depart here from the work of
Jeffery by including this effect at the level of the Stokes equation
on $u$. The extra stress due to particle $i$ can be thought as a
dipole: the sum of two opposite and close point forces, along the
direction of $p_i$: typically, for bacteria, one point force is
centered at the body of the particles, while the other is centered at
the flagella, see \cref{fig:pusher-puller} distinguishing between two
kinds of bacteria: pushers and pullers. One ends up with
$$ -\Delta u + \nabla q = \alpha_N \sum_{i=1}^N \nabla_x \cdot  \left(  ( p_i \otimes p_i ) \delta_{x_i} \right), \quad \nabla_x \cdot   u= 0,$$
where this appearance of the divergence operator in front of a Dirac
mass is typical of a dipole. Depending on the orientation of the
dipole, the sign of $\alpha_N$ can be positive (pullers) or negative
(pushers). We refer to \cite{ChenLiu2013} for a more rigorous
derivation.

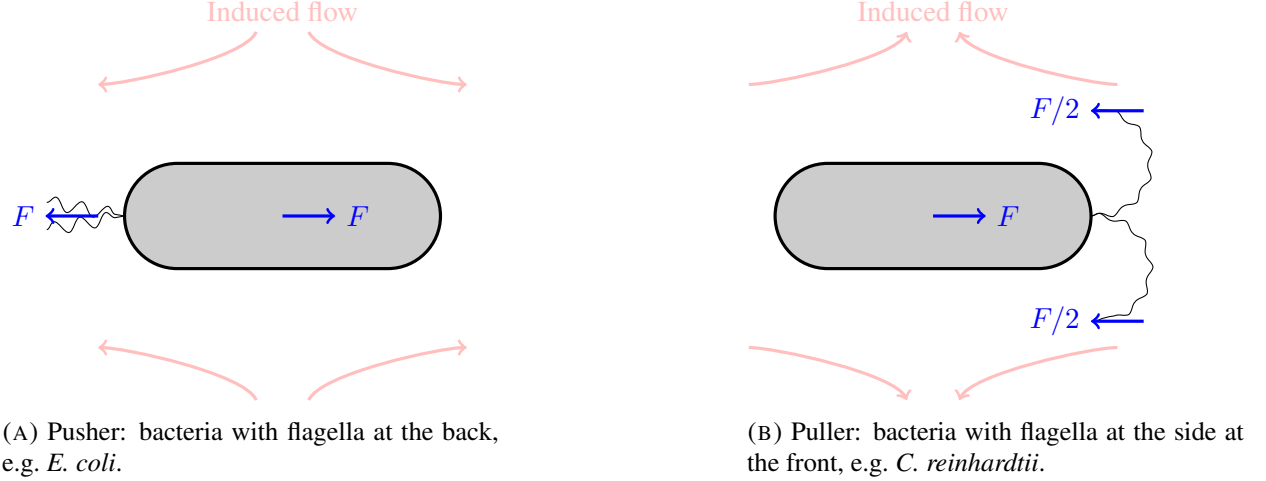
\begin{figure}[tb]
  \centering
  \begin{subfigure}{0.4\textwidth}
    \begin{tikzpicture}[scale=0.7]
      \draw[very thick, black,fill=black!20!white]
      (-2,1) -- (2,1) arc (90:-90:1)
      -- (2,-1) -- (-2,-1) arc (-90:-270:1);
      \draw[decorate,decoration=snake] (-3,0) + (170:1.5) -- (-3,0);
      \draw[decorate,decoration=snake] (-3,0) + (190:1.5) -- (-3,0);
      \draw[very thick,blue,->] (0,0) -- (1,0) node[anchor=west] {\(F\)};
      \draw[very thick,blue,->] (-3.5,0) -- (-4.5,0) node[anchor=east] {\(F\)};
      \draw[very thick,pink,->] (0.5,3.5) .. controls (0.7,3) and (3,2.5) .. (3.5,2.5);
      \draw[very thick,pink,->] (0.5,-3.5) .. controls (0.7,-3) and (3,-2.5) .. (3.5,-2.5);
      \draw[very thick,pink,->] (-0.5,3.5) .. controls (-0.7,3) and (-3,2.5) .. (-3.5,2.5);
      \draw[very thick,pink,->] (-0.5,-3.5) .. controls (-0.7,-3) and (-3,-2.5) .. (-3.5,-2.5);
      \draw[pink] (0,3.5) node[anchor=south] {Induced flow};
    \end{tikzpicture}
    \caption{Pusher: bacteria with flagella at the back, e.g.\
      \textit{E. coli}.}
  \end{subfigure}
  \hfill
  \begin{subfigure}{0.4\textwidth}
    \begin{tikzpicture}[scale=0.7]
      \draw[very thick, black,fill=black!20!white]
      (-2,1) -- (2,1) arc (90:-90:1)
      -- (2,-1) -- (-2,-1) arc (-90:-270:1);
      \draw[decorate,decoration={snake,amplitude=1}] (3,0) arc (-90:90:1);
      \draw[decorate,decoration={snake,amplitude=1}] (3,0) arc (90:-90:1);
      \draw[very thick,blue,->] (0,0) -- (1,0) node[anchor=west] {\(F\)};
      \draw[very thick,blue,->] (4,2) -- (3,2) node[anchor=east] {\(F/2\)};
      \draw[very thick,blue,->] (4,-2) -- (3,-2) node[anchor=east] {\(F/2\)};
      \draw[very thick,pink,<-] (0.5,3.5) .. controls (0.7,3) and (3,2.5) .. (3.5,2.5);
      \draw[very thick,pink,<-] (0.5,-3.5) .. controls (0.7,-3) and (3,-2.5) .. (3.5,-2.5);
      \draw[very thick,pink,<-] (-0.5,3.5) .. controls (-0.7,3) and (-3,2.5) .. (-3.5,2.5);
      \draw[very thick,pink,<-] (-0.5,-3.5) .. controls (-0.7,-3) and (-3,-2.5) .. (-3.5,-2.5);
      \draw[pink] (0,3.5) node[anchor=south] {Induced flow};
    \end{tikzpicture}
    \caption{Puller: bacteria with flagella at the side at the front, e.g.\
      \textit{C. reinhardtii}.}
  \end{subfigure}
  \caption{Illustration of pusher and puller active particles moving
    to the right. Depending on configuration, the force for propulsion
    is differently applied on the fluid, which yields a different
    induced flow around the active particle \cite{saintillan-2018-rheology}.}
  \label{fig:pusher-puller}
\end{figure}

Eventually, with the right scaling of $\alpha_N$, and performing a (formal)
mean-field limit, one obtains the Saintillan-Shelley model, with two
unknowns:
\begin{itemize}
\item $\Psi = \Psi(t,x,p)$, the distribution of particles in space and
  orientation,
\item  $u = u(t,x)$, the velocity field of the fluid.
\end{itemize}
It reads
\begin{equation} \label{SS}
  \begin{aligned}
    &\partial_t \Psi
      +(U_0p +u)\cdot \nabla_x\Psi
      + \nabla_p\cdot \Big(\mathbb{P}_{p^\perp}
      \left[(\gamma E(u) + W(u))p\right] \Psi\Big)
      = \nu \Delta_p \Psi,\\
    &-\Delta_xu +\nabla_x q= \alpha \nabla_x \cdot \int_{\bbS^2} \Psi(t,x,p)\, p\otimes p \, \dd p,\\
    &\nabla_x \cdot u=0,
  \end{aligned}
\end{equation}
where we remind that
\begin{align}
  E(u)=\frac12 \left[\nabla u + (\nabla u)^T\right] , \qquad
  W(u)=\frac12 \left[\nabla u - (\nabla u)^T\right].
\end{align}
The first equation on $\Psi$ reflects transport along the stream lines
associated to
\begin{equation*}
  \dot{x} = u(x) + U_0 p, \quad
  \dot{p} = \mathbb{P}_{p^\perp} \left[(\gamma E(u) + W(u))(x)p\right]
\end{equation*}
inspired by the dynamics of the particles mentioned above. It also
contains some diffusion term $\nu \Delta_p \Psi$ ($\nu \ll 1$), where
$\Delta_p$ refers to the Laplace-Beltrami operator associated to
possible Brownian rotational noise. The other equations describe the
Stokes flow, which is forced by the extra stress
$\operatorname{div} \sigma$, with $\sigma$ given by the marginal in
$p$ of $\alpha \Psi p \otimes p$.

Let us notice that this model is very close to the Doi model
\cite{DoiEdwards88} which describes passive suspensions of rodlike
polymers. In this setting, $U_0 = 0$ (passive particles), and
$\alpha > 0$ is proportional to the Boltzmann constant and the
temperature: the tensor $\sigma$ models viscoelastic stress.  While
local existence and uniqueness of smooth solutions of both the Doi and
the Saintillan-Shelley models can be obtained by standard methods, the
question of global in time well-posedness is harder. In the case of
the Doi model, global well-posedness was proved in
\cite{OttoTzavaras2008} for $\nu > 0$, and in \cite{Constantin2005}
for $\nu \ge 0$. Actually a look at the proof of \cite{Constantin2005}
shows that it is still valid for $\alpha$ of any sign. However, for
non-zero $U_0$, as far as we know, global well-posedness is unknown
(except for the addition of diffusion in $x$, see
\cite{ChenLiu2013}). It seems to be an interesting open problem,
especially when $\nu = 0$: indeed, one can check that the force field
in the equation for
$\mathbb{P}_{p^\perp} \left[(\gamma E(u) + W(u))p\right]$ has the same
regularity in $x$ as $\sigma$. This is to be compared with
Vlasov-Poisson equation, where the force field has one degree of
regularity more than the density $\rho$. In the present paper, we will
not contribute to this well-posedness theory, but will rather
investigate qualitative properties at the linear level.

\subsection{Asymptotic stability of incoherence: main results}
A class of equilibria of interest is given by
\begin{align}
  \Psi^S=\Psi^S(p), \qquad u^S= 0
\end{align}
of which the isotropic suspension $\Psi^{iso}=1/4\pi$ is a particular
case. By analogy with the Kuramoto model, we will call \(\Psi^{iso}\)
the incoherent state, as it reflects zero alignment between the
orientation of the particles. The linearized system around
$(\Psi^{iso}, u^S = 0)$ reads
\begin{align}
  &\partial_t \psi +U_0 p\cdot \nabla_x\psi  + \frac{1}{4\pi} \nabla_p\cdot
    \Big(\mathbb{P}_{p^\perp} \left[(\gamma E(u) + W(u))p \right] \Big)
    =  \nu \Delta_p \psi,\\
  &-\Delta_x u +\nabla_x q= \nabla_x \cdot \alpha \int_{\bbS^2} \psi(t,x,p)\, p\otimes p \,  \dd p,\\
  &\nabla_x \cdot u=0.
\end{align}
Since $E(u)$ is symmetric, $W(u)$ is skew-symmetric and $u$ is divergence-free, we have
\begin{align}
  &\nabla_p\cdot \left((\Id-p\otimes p)   E(u) p  \right)
    =-3p\otimes p:E(u),\qquad \nabla_p\cdot \left((\Id-p\otimes p)   W(u)p  \right)=0.
\end{align}
Hence the equations become
\begin{equation} \label{LSS}
  \begin{aligned}
    &\partial_t \psi +U_0p \cdot \nabla_x\psi -\frac{3\gamma}{4\pi}p\otimes p:E(u)= \nu \Delta_p  \psi,\\
    &-\Delta_x u +\nabla_x q= \nabla_x \cdot \alpha \int_{\bbS^2} \psi(t,x,p)\, p\otimes p \, \dd p,\\
    &\nabla_x \cdot u=0.
  \end{aligned}
\end{equation}
A partial analysis of system \eqref{LSS} is performed in \cite{saintillan-shelley-2008-instabilities,HoheneggerShelley2010}

\begin{itemize}
\item in the case $\nu = 0$, looking for unstable eigenmodes, the
  authors manage to calculate an explicit dispersion relation. They
  show that in the case of suspensions of pullers ($\alpha > 0$), no
  unstable eigenvalue exists. As pointed out by the authors, this is
  consistent with the fact that for the full nonlinear system
  \eqref{LSS}, for any $\alpha > 0$, $\nu \ge 0$, the relative entropy
  of any solution $\psi$ with respect to $\psi^{iso}$ decays. However,
  the situation changes completely for pushers ($\alpha < 0$). In this
  case, unstable eigenvalues exist at low enough $x$ frequencies. In
  other words, the length $L$ of the torus is a bifurcation parameter:
  there is a critical value $L_c$, computed numerically in
  \cite{saintillan-shelley-2008-instabilities}, such that for
  $L < L_c$, there exists no unstable eigenvalue, while for $L > L_c$,
  there exists at least one. Above this threshold, the incoherent
  state loses its stability, some partial alignment of the ellipsoidal
  particles is observed numerically, while the corresponding velocity
  field develops patterns that are favourable to mixing of passive
  scalars advected by the flow. See the nice recent work
  \cite{Ohm_Shelley_bifurcation} for more on the nature of the
  bifurcation process.

\item in the case of small $\nu >0$, no explicit dispersion relation
  is available. But numerical computations in
  \cite{saintillan-shelley-2008-instabilities}, confirm the picture
  given at $\nu = 0$, with some threshold close to $L_c$. This
  numerical observation will be confirmed rigorously here.
\end{itemize}

Our focus in the present paper is on what we call the incoherent
regime, that is $L < L_c$, both in the case $\nu=0$ and $\nu >
0$. Again, some theoretical observations are already contained in
\cite{saintillan-shelley-2008-instabilities}, see also \cite{HoheneggerShelley2010}. For $\nu =0$, the absence of unstable eigenvalue for the
linearized operator in \eqref{LSS} does not imply automatically linear
stability, due to the possible unstable continuous spectrum.
Numerical simulations in \cite{HoheneggerShelley2010} show high frequency oscillations but
no instability. Moreover, some decay at rate $(kt)^{-2}$, where $k$ is
the $x$-frequency of the perturbation, is seen for some moment of the
solution $\psi$ with respect to $p$. Stability of the incoherent state
is conjectured on the basis of these simulations. Regarding $\nu > 0$, stability is confirmed by simulations, but no clear rate of convergence with
respect to $\nu$ is given.

Our aim is to clarify most of these aspects, through a detailed
mathematical study of the linearized equation \eqref{LSS}. Since
$x \in \TT^3_L$, we take the Fourier transform of $\psi$
\begin{align}
  &\psi_{k}(t,p)
    = \int_{ \TT^3_L} \e^{- i  k \cdot x}\psi(t,x,p)\, \dd x,
    \qquad k \in \frac{2\pi}{L}  \ZZ^3,\\
  &\psi(t,x,p)=\frac{1}{L^3}\sum_{k\in \frac{2\pi}{L}\ZZ^3}  \psi_{k}(t,p)\, \e^{i k\cdot x},
\end{align}
and similarly for $u_k$, the Fourier transform of $u$. The Fourier transform of the equation is then
\begin{equation*}
  \begin{aligned}
    &\partial_t \psi_k+U_0 i k\cdot p \,\psi_k=\frac{3\gamma}{4\pi}p\otimes p:E_k(u_k) + \nu \Delta_p \psi_k,\\
    &u_k := \frac{i}{k^2}  \mathbb{P}_{k^\perp}  \sigma_k k, \quad
      \mathbb{P}_{k^\perp} := \left(\Id- \frac{k}{|k|} \otimes \frac{k}{|k|} \right), \quad E_k(u) := \frac{i}{2} (k \otimes u + u \otimes k),\\
    &\sigma_k := \alpha \int_{\bbS^2} \psi_k(t,p)\, p\otimes p\, \dd p.
  \end{aligned}
\end{equation*}
Note that for $k=0$, the equation reduces to the simple heat equation
\begin{equation*}
  \partial_t \psi_0 - \nu \Delta_p \psi_0  = 0
\end{equation*}
so that we restrict to $k \neq 0$. Moreover, through the change of
variables
\begin{equation}\label{eq:rescaling}
  t := \frac{t}{U_0 |k|},
  \quad u := \frac{u}{|\alpha k|},
  \quad \Gamma := \frac{\gamma |\alpha|}{U_0 |k|},
  \quad \nu := \frac{\nu}{U_0 |k|},
  \quad k := \frac{k}{|k|},
  \quad \eps := \frac{\alpha}{|\alpha|},
\end{equation}
the system becomes
\begin{equation} \label{LSS-mode}
  \begin{aligned}
    &\partial_t \psi_k +  i k\cdot p \,\psi_k= \frac{3 \Gamma}{4\pi} p\otimes p:E(u_k) + \nu \Delta_p \psi_k,\\
    &u_k = i  \mathbb{P}_{k^\perp}  \sigma_k k, \quad \mathbb{P}_{k^\perp} = \left(I- k \otimes k \right),\\
    &\sigma_k = \eps \int_{\bbS^2} \psi_k(t,p)\, p\otimes p\, \dd p,
  \end{aligned}
\end{equation}
where
\begin{equation}
  k \in \bbS^2, \quad \eps = \pm 1 \quad \text{ with $\eps = 1$ for pullers, $\eps = -1$ for pushers.}
\end{equation}

In this rescaled setting, our main results are the following:
\begin{theorem}[\textbf{Inviscid case, pointwise decay through mixing}] \label{thm1}

  Let $\nu = 0$,
  $\psi_k^{in} = \psi_k^{in}(p) \in H^{3+\delta}(\bbS^2)$ for some
  $\delta > 0$, \(k \in \bbS^2\), \(\eps=\pm 1\) and
  \(\Gamma \in \RR_+\).  Let $\psi_k = \psi_k(t,p)$ be the solution of
  \eqref{LSS-mode} such that $\psi_k\vert_{t=0} = \psi_k^{in}$. There
  is an absolute constant $\Gamma_c \in (0,+\infty]$, with
  $\Gamma_c = +\infty$ for $\eps = 1$, and
  $C = C\big(\delta,\Gamma\big)$, such that if $\Gamma < \Gamma_c$,
  then for all $t \ge 0$
  \begin{align}
    \label{eq:inviscid-u-decay}
    |u_k(t)| &\le  \frac{C}{(1+t)^2}   \| \psi^{in}_{k} \|_{H^{3+\delta}(\bbS^2)},\\
    \label{eq:inviscid-psi-decay}
    \| \psi_k(t, \cdot)  \|_{H^{-1-\delta}(\bbS^2)}
             &\le  \frac{C}{(1+t)}  \| \psi^{in}_{k} \|_{H^{3+\delta}(\bbS^2)}.
  \end{align}
\end{theorem}

\begin{theorem}[\textbf{Diffusive case, mixing persists for small $\nu$ up to time $\nu^{-1/2}$}]\label{thm2}

  Let $\nu \in (0,1)$,
  $\psi_k^{in} = \psi_k^{in}(p) \in H^{2+\delta}(\bbS^2)$ for some
  $\delta > 0$, \(k \in \bbS^2\), \(\eps=\pm 1\) and
  \(\Gamma \in [0, \Gamma_c)\), with $\Gamma_c$ as in \cref{thm1}.
  Let $\psi_k = \psi_k(t,p)$ be the solution of \eqref{LSS-mode} such
  that $\psi_k\vert_{t=0} = \psi_k^{in}$.  There exist
  $\nu_0 = \nu_0(\Gamma) > 0$, $C = C(\delta, \Gamma) > 0$, such that
  for $0 < \nu < \nu_{0}$ and \(t \in [0,\nu^{-1/2}]\) it holds that
  \begin{align}
    |u_k(t)| &\le  \frac{C \ln(2+t)}{(1+t)^2}   \| \psi^{in}_{k} \|_{H^{2+\delta}(\bbS^2)},\\
    \| \psi_k(t, \cdot)  \|_{H^{-1-\delta}(\bbS^2)}
             &\le   \frac{C \ln(2+t)}{(1+t)}  \| \psi^{in}_{k} \|_{H^{2+\delta}(\bbS^2)}.
  \end{align}
\end{theorem}

After the time \(\nu^{-1/2}\), we have the mixing estimates with an
additional log factor until the enhanced dissipation takes over.

\begin{theorem}[\textbf{Diffusive case, enhanced dissipation at timescale $\nu^{-1/2}$}]\label{thm3}
  Under the assumptions of \cref{thm2}, there exist
  $M, \eta, \nu_0 > 0$ depending on $\Gamma$, and
  $C = C(\delta,\Gamma)$ such that for $0 < \nu < \nu_0$ and
  $t \ge \nu^{-1/2}$ it holds that
  \begin{align}
    |u_k(t)|
    &\le C \min \Big( |\ln(\nu)|^M \nu ,  \e^{-\eta \nu^{1/2} t} \Big)
      \| \psi^{in}_{k} \|_{H^{2+\delta}(\bbS^2)}, \\
    \| \psi_k(t)  \|_{L^2(\bbS^2)}
    &\le   C (1+t) \e^{-\eta \nu^{1/2} t}  \| \psi^{in}_{k} \|_{H^{2+\delta}(\bbS^2)}.
  \end{align}
\end{theorem}
A few remarks are in order.
\begin{remark}
  On the torus the largest length scale is \(L\), so that in the
  unscaled equation the original wavenumber $|k|$ is at least
  $2\pi/L$.  This leads to the upper bound
  \begin{equation}
    \Gamma \le \frac{\gamma |\alpha| L}{2\pi U_0}.
  \end{equation}
  Hence the stability condition on \(\Gamma\) can be translated into a
  maximal size of the torus.
\end{remark}
\begin{remark}
  \Cref{thm1} is the expression of an inviscid damping: it leads to a
  $O(t^{-2})$ decay for the velocity field $u_k$. Notice that $u_k$
  involves $\psi_k$ through an average in $p$: this average allows to
  take advantage of the mixing phenomenon. As a byproduct, one has a
  $O(t^{-1})$ decay of $\psi_k$ itself in weak topology (negative
  Sobolev space). The difference in the rate of decay is very much
  related to the special structure of $u_k$ in terms of $\psi_k$. We
  stress that these polynomial rates cannot be improved, even taking
  analytic initial data. This is a major difference with
  Landau damping, and is related to the fact that the orientation
  variable $p \in \bbS^2$ replaces the velocity variable
  $v \in \RR^3$.
\end{remark}

\begin{remark} \label{rem_gamma_c}

  In our proof of \cref{thm1}, the stability constant $\Gamma_c$ that
  we exhibit when $\eps = -1$ is sharp: it means that for
  \(\Gamma > \Gamma_c\) there exist eigenmodes with \(u_k\)
  growing. Concretely, condition $\Gamma < \Gamma_c$ is equivalent
  to the fact that some dispersion relation has no root in the
  unstable half-plane:
  \begin{align} \label{dispersion_relation}
    F_\gamma(z) \neq 0, \quad \quad \forall \Re(z) \ge 0.
  \end{align}
  Our rigorous analysis of this dispersion relation, including the
  identification of $\Gamma_c$, is inspired by the work of Penrose
  \cite{penrose-1960-electrostatic-maxwellian} on the stability of
  plasmas. It substitutes for the numerical analysis carried in
  \cite{saintillan-shelley-2008-instabilities}.
\end{remark}

\begin{remark}\label{rem_nu}
  \Cref{thm2}, dealing with the weakly diffusive setting, shows that
  for $\nu > 0$ small enough, the mixing phenomenon persists up to
  time $\nu^{-1/2}$, more generally up to time $C \nu^{-1/2}$ for
  arbitrary $C > 0$.  The fact that $\nu^{-1/2}$ is a natural time
  threshold for our problem is confirmed by \cref{thm3}: solutions of
  the equation experience an exponential decay after this typical
  threshold. This exponential decay before the natural diffusive
  timescale $\nu^{-1}$ reflects the well-known phenomenon of enhanced
  dissipation \cite{BCZ17,CKRZ08, CZDE20, CZG21}, evoked in the
  introduction, although one can notice again a strong qualitative
  difference between variable $p \in \bbS^2$ and variable
  $v \in \RR^3$. In the latter case, the enhanced dissipation would
  hold with typical time scale $\nu^{-1/3}$. Here, the typical time
  scale is much longer, which creates strong mathematical difficulties
  in showing mixing up to this time scale, and enhanced diffusion
  afterwards.

  For the transition time until the enhanced dissipation takes over,
  we show the persistence of the mixing estimates with additional log
  factors. In advection-diffusion problems, a diffusion-limited mixing
  behavior is often observed \cite{MD18}. That is, inviscid mixing
  does not persist as $t\to\infty$, rather the ratio of the
  $\dot{H}^{-1}$ to the $L^2$ norm converges to a positive
  constant. This corresponds to the existence of a characteristic
  filamentation length scale, often referred to as the Batchelor scale
  in the physics literature.
 \end{remark}

\begin{remark}

  While completing our manuscript, we noticed the release of the
  interesting preprint \cite{albritton-ohm-2022-preprint}, about the
  same Saintillan-Shelley model (except for the introduction of an
  additional diffusion in variable $x$). Although paper
  \cite{albritton-ohm-2022-preprint} and ours share a few common
  features, we believe that they are different enough to provide
  distinct and valuable insights into the stability of active
  suspensions.

  Regarding mixing, which is the focus of the present paper,
  \cite{albritton-ohm-2022-preprint} only contains a weaker version of
  our \cref{thm1}, showing that under condition
  \eqref{dispersion_relation}, a $L^2$ in time stability estimate
  $$ \int_{\RR_+} |u_k(t)|^2 (1+t)^{3-\epsilon} \dd t< \infty $$
  holds. Our extensive discussion of \eqref{dispersion_relation}, see
  \cref{rem_gamma_c}, as well as our optimal pointwise in time decay
  estimates are not covered. We stress that these pointwise estimates
  are based on the general \cref{thm_decay_Volterra} on
  Volterra equations, which is of independent interest.

  More importantly, \cite{albritton-ohm-2022-preprint} does not
  contain any analogue of our \cref{thm2}, which is the heart
  of our paper, and requires completely different arguments from the
  inviscid case.

  On the other hand, \cite{albritton-ohm-2022-preprint} contains the
  derivation of Taylor dispersion estimates when $x\in \RR^d$, and two
  nonlinear stability results (diffusion in $x$ is crucial
  there). First, in the case of pullers ($\eps = 1$), using an
  approach \`a la Guo~\cite{guo-2002-landau}, the authors prove
  nonlinear stability of the incoherent state $\Psi^{iso}$, both for
  $x \in \TT^d$ and $x \in \RR^d$ (but without enhanced
  dissipation). Second, in the case of pushers, they prove nonlinear
  stability of the incoherent state with enhanced diffusion, under the
  stringent assumption $\Gamma = o(\nu^{1/2})$, which allows to treat
  the model as a perturbation of the advection-diffusion equation.
  With regards to these last two results, and although we restrict
  here to a linear analysis, our \cref{thm3} is a significant
  progress.
\end{remark}

\subsection{Key ideas}
\label{sec:intro:key-ideas}

The evolution \eqref{LSS-mode} can be split as
\begin{equation}
  \label{eq:split-l1-l2}
  \partial_t \psi_k = L_{1,k} \psi_k + \bar L_{2,k} \cdot u_k[\psi_k]
\end{equation}
where
\begin{equation} \label{def_L1}
L_{1,k}:= -ik\cdot p + \nu \Delta_p
\end{equation}
is the advection-diffusion operator, and \(u_k\) is the
low-dimensional linear map from the kinetic distribution to the
macroscopic velocity field given by
\begin{equation} \label{def_uk}
  \begin{aligned}
    u_k[\psi]
    :=   i  \mathbb{P}_{k^\perp}  \sigma k, \quad
    \mathbb{P}_{k^\perp} = \left(I- k \otimes k \right), \quad
    \sigma  := \eps \int_{\bbS^2} \psi(p)\, p\otimes p\, \dd p.
  \end{aligned}
\end{equation}
Finally, $\bar L_{2,k}$ is the vector field independent of time
defined by
\begin{equation} \label{def_L2}
  \bar L_{2,k}(p)  := \frac{3 i \Gamma}{4\pi} (p \cdot k) \mathbb{P}_{k^\perp} p.
  \end{equation}
  Here, we have used that $u_k$ satisfies $k \cdot u = 0$ so that
  $$ \frac{3\Gamma}{4\pi} p \otimes p : E_k(u_k) =  \bar L_{2,k} \cdot u_k. $$
By Duhamel's formula
\begin{equation}
  \label{eq:duhamel-psi}
  \psi_k(t, \cdot) = \e^{tL_{1,k}} \psi_k^{in}
  + \int_0^t \e^{(t-s)L_{1,k}} \left( \bar L_{2,k} \cdot u_k[\psi_k(s, \cdot)] \right)\, \dd s.
\end{equation}
Applying the linear map \(u_k\) then yields the Volterra
equation
 \begin{equation}
  \label{eq:volterra-u}
   u_k(t) + \int_0^t K_k(t-s)\, u_k(s)\, \dd s = U_k(t)
  \end{equation}
  with the source
\begin{equation}
  \label{eq:volterra-source}
  U_k(t) = u_k[\e^{tL_{1,k}} \psi_k^{in}]
\end{equation}
and the matrix kernel $K_k$ defined by
\begin{equation}
  \label{eq:volterra-kernel}
  K_k(t) v = - u_k[\e^{tL_{1,k}} (\bar L_{2,k} \cdot v)], \quad v \in \CC^3.
\end{equation}
In the proof of Theorems \ref{thm1}, \ref{thm2} and \ref{thm3} the
key point is to work on the Volterra equation \eqref{eq:volterra-u} and
to show that $u_k$ decays like $t^{-2}$ at infinity. Once this is
obtained, to go back to the evolution of \(\psi_k\) and estimate its
decay in negative Sobolev norms is direct. Regarding the decay of
$u_k$, there are two main steps:
\begin{itemize}
\item[i)] To show that the source term $U_k$ and the kernel $K_k$
  decay like in the corresponding theorem.
\item[ii)] To show that under appropriate conditions on $\Gamma$, this
  decay passes to the solution $u_k$ of the Volterra equation.
\end{itemize}
Let us stress that the mathematical questions raised by the second
step are not specific to our model, and can be asked for any Volterra
equation. They will be examined in \cref{sec:Volterra}. Obviously,
before relating the decay rate of the solution to the decay rates of
the source and the kernel, a prerequisite is the stability of the
solution. Using the classical theory of Volterra equation, one knows
that it is equivalent to the spectral condition
\begin{equation} \label{spectral_Volterra}
 \det(I + \mathcal{L}K_k(z)) \neq 0, \quad \forall \Re(z) \ge 0,
\end{equation}
where $\mathcal{L}$ is the Laplace transform.  Under this condition,
we achieve Step ii) by proving \cref{thm_decay_Volterra} in
\cref{sec:Volterra}. We provide a short proof, that uses the notion of
Volterra kernel of type $L^\infty$ and the underlying structure of
Banach algebra of this class of kernels

As regards Step i), there is a main difference between $\nu=0$ and
$\nu > 0$. In the inviscid case \(\nu=0\), considered in Section
\ref{sec:inviscid}, the evolution \(\e^{tL_{1,k}}\) can be solved
explicitly: the kernel and sources are given by Fourier transforms on
the sphere, whose decay properties are well-known, and analyzed
through stationary phase arguments. This implies that \(U_k\) and
\(K_k\) are \(O(t^{-2})\). Moreover, the spectral stability condition
can be fully understood through an analysis \`a la Penrose, see
\cref{sec:inviscid}, completing the proof of \cref{thm1}.

For the diffusive case \(\nu>0\), we cannot solve the evolution
\(\e^{tL_{1,k}}\) explicitly and the main effort in proving
Theorem~\ref{thm2} and \ref{thm3} is to obtain optimal mixing
estimates on the advection-diffusion equation. This is performed in
\cref{sec:diffusive}. In terms of the advection-diffusion evolution
\(\e^{tL_1}\), the first result is an enhanced dissipation result,
which can be proved by adapting hypocoercive methods:
\begin{proposition}\label{prop:hypocoercf}
  There exist constants $C_0,\eps_0,\nu_0>0$ with the following
  property: for every $\nu_0>\nu>0$ and \(\psi^{in}\) the
  advection-diffusion evolution is bounded for all \(t \ge 0\) as
  \begin{align}\label{eq:enhanceddissip}
    \| \e^{t L_{1,k}} \psi^{in}\|_{L^2(\bbS^2)}
    \leq   C_0 \e^{-\eps_0 \nu^{1/2} t}
    \| \psi^{in} \|_{L^2(\bbS^2)}.
  \end{align}
\end{proposition}
The main novelty is that the inviscid mixing estimates persists until
the enhanced dissipation takes over:
\begin{proposition}\label{thm:vector-field-decay}
  There exist $\nu_0 > 0$ and \(M\ge 2\) such that for \(c,\delta > 0\)
  there exists a constant \(C\) with the following bounds: For any \(\nu \le \nu_0\),
  any scalar function $F$ and
  $t \le \nu^{-1/2}$ it holds that
  \begin{equation}
    \left| \int_{\bbS^2} (\e^{tL_{1,k}} \psi^{in}) F\, \dd p\right|
    \le C \frac{\sqrt{\ln (2+t)}}{(1+t)} \|F\|_{H^{1+\delta}} \|\psi^{in}\|_{H^{1+\delta}}
  \end{equation}
  and
  \begin{equation}
    \left| \int_{\bbS^2} (\e^{tL_{1,k}} \psi^{in}) F\, \nabla(p \cdot k)\, \dd p\right|
    \le C \frac{\sqrt{\ln (2+t)}}{(1+t)^2} \|F\|_{H^{2+\delta}}
    \|\psi^{in}\|_{H^{2+\delta}}.
  \end{equation}
  Further, for \(t \in [\nu^{-1/2},c\nu^{-1/2}|\ln \nu|]\) it holds
  that
  \begin{equation}
    \left| \int_{\bbS^2} (\e^{tL_{1,k}} \psi^{in}) F\, \dd p\right|
    \le C \nu^{1/2} |\ln \nu |^M  \|F\|_{H^{1+\delta}} \|\psi^{in}\|_{H^{1+\delta}}
  \end{equation}
  and
  \begin{equation}
    \left| \int_{\bbS^2} (\e^{tL_{1,k}} \psi^{in}) F\, \nabla(p \cdot k)\, \dd p\right|
    \le C \nu  |\ln \nu|^M    \|F\|_{H^{2+\delta}}
    \|\psi^{in}\|_{H^{2+\delta}}.
  \end{equation}
\end{proposition}

The key idea is the use of the vector field method. This method,
introduced for the analysis of decay properties of wave equations
\cite{Klainerman85}, was used recently in the context of Vlasov type
equations on $\RR^d$ \cite{CLN21}. It allowed to exhibit mixing when
$0 < \nu \ll 1$, despite the loss of explicit representation formula.
Let us just mention at this stage that the key point is to construct
good vector fields, meaning that they commute to the
advection-diffusion operator $i k \cdot p - \nu \Delta_p$. In the
Euclidean setting~\eqref{eq:euclidean-free-transport}, the natural
choice $J = \nabla_v + i t k$ works well over the time scale
\(\nu^{-1/3}\), after which enhanced dissipation dominates. On the
sphere this has no good analogue as the obvious generalization
$J = \nabla_p + i \nabla(p\cdot k) t$, when applied to a solution
$\psi$ of the advection-diffusion equation, creates commutators of the
form $\nu t \nabla \psi$, with a time integral that can not be
controlled over the large time scale $\nu^{-1/2}$. We overcome this
difficulty by combining two ideas. We first introduce a better vector
field,
\begin{equation*}
  J_\nu =   \alpha(t) \nabla + i \beta(t) \nabla(p\cdot k), \quad \text{where} \quad
  \alpha(t): = \cosh(\sqrt{-2i\nu}\, t), \quad
  \beta(t): = \frac{1}{\sqrt{-2i\nu}} \sinh(\sqrt{-2i\nu}\, t).
\end{equation*}
Roughly, it allows to replace the bad commutator by one that
behaves like $\nu t (1-k \cdot p)\nabla \psi$, hence vanishing near
$\pm k$. Then, we adapt Villani's hypocoercivity method
\cite{villani-2009-hypocoercivity}, using additional weights in
time. The key point of this adaptation is that, besides proving
enhanced diffusion at time scale $\nu^{-1/2}$, see
\cref{rem_nu}, it provides extra decay information for some quantities
vanishing near the poles $\pm k$ of the sphere. This allows to control
the commutator and apply the vector field method. All details will be
found in \cref{sec:diffusive}.

Using the decay estimate, the proof of \cref{thm2} is achieved in
\cref{sub:conclusion:mixing}.  Regarding the enhanced dissipation
estimates of \cref{thm3}, a key point is to understand the behaviour of
the diffusive Volterra kernel past time $\nu^{-1/2}$, and notably to
check the spectral condition \eqref{spectral_Volterra} for small
$\nu$. In the Euclidean case, see for instance \cite{CLN21}, this can
be proved perturbatively from the inviscid case, using the decay of
the diffusive kernel uniformly in $\nu$. In our spherical setting,
this is where we need the second part of
\cref{thm:vector-field-decay}. This analysis and the proof of
\cref{thm3} are performed in \cref{sub:conclusion:enhanced}.

\section{Volterra equations} \label{sec:Volterra}

In this section we study a general Volterra equation:
\begin{align} \label{classic_Volterra1}
 u(t)  + K \star u(t) = v(t), \quad t \ge 0,
\end{align}
with unknown $u : \RR_+ \rightarrow \CC^n$, and data
$K : \RR_+ \rightarrow M_n(\CC)$ (the kernel),
$v : \RR_+ \rightarrow \CC^n$ (the source). The convolution is here on
$\RR_+$, defined by $F \star g(t) = \int_0^t F(s)\, g(t-s)\, \dd s$.
We are interested in the global in time solvability of this equation,
and in accurate polynomial decay estimates for solution $u$ under
assumptions of polynomial decay on $K$ and $v$.  The classical path to
study this equation, detailed in
\cite{gripenberg-londen-staffans-1990-volterra}, is to construct the
so-called resolvent of the equation, that is a matrix-valued
$R : \RR_+ \rightarrow M_n(\CC)$ satisfying
\begin{align}\label{classic_Volterra2}
 R(t)  + K \star R(t) =   R(t)  +  R  \star K(t) = K(t).
\end{align}
Note that the convolution product does not commute when $n > 1$, so
that in this case, the resolvent satisfies two distinct equations. It
is then straightforward to check that if $K$ and $v$ are integrable,
and if there exists an integrable solution $R$ of
\eqref{classic_Volterra2}, then an integrable solution $u$ of
\eqref{classic_Volterra1} is given by $u = v - R \star v$ and is
unique.

For the construction of the resolvent, a vague idea of proof is the
following. Assuming there is a solution $R$ to
\eqref{classic_Volterra2}, and extending both $R$ and $K$ by zero on
$\RR_-$, the relation \eqref{classic_Volterra2} holds now for all $t$
on $\RR$, replacing the convolution on $\RR_+$ by the usual
convolution on $\RR$. Taking the Fourier transform yields in
particular
\begin{align} \label{Volterra_Fourier}
 (I + \hat{K}(\xi))   \hat{R}(\xi) = \hat{K}(\xi), \quad \forall \xi \in \RR.
 \end{align}
This suggests to impose the condition
\begin{align} \label{naive_spectral}
 \det (I + \hat{K}) \neq 0 \quad \text{ on } \: \RR
\end{align}
and to define
\begin{align} \label{formula_resolvent}
  R = \mathcal{F}^{-1} (I + \hat{K})^{-1} \hat{K}.
\end{align}
This is however too quick. We remind that, in order to obtain
\eqref{Volterra_Fourier}, we have extended our hypothetical solution $R$
by zero on $\RR_-$. Hence, we must not only verify that the formula
\eqref{formula_resolvent} makes sense, but also that it gives a
function that vanishes on $\RR_{-}$. Condition
\eqref{Volterra_Fourier} is not sufficient for that. Still, under the
stronger condition
\begin{align} \label{spectral}
 \det (I + \cL K(z)) \neq 0 \quad \text{ for all } \: \Re z \ge 0,
\end{align}
where $\cL K(z) = \int_0^{+\infty} \e^{- zt} K(t)\, \dd t$ is the
Laplace transform of the kernel, everything goes nicely. This is the
content of following theorem.
\begin{theorem}[{\cite[Chapter~2]{gripenberg-londen-staffans-1990-volterra}}]\label{thm:existVolterra}
  Assume that $K \in L^1(\RR_+, M_n(\CC))$, and that the spectral
  condition \eqref{spectral} is satisfied. Then, there exists a unique
  solution $R \in L^1(\RR_+, M_n(\CC))$ of
  \eqref{classic_Volterra2}. As a consequence, for any
  $v \in L^1(\RR_+, \CC^n)$, \eqref{classic_Volterra1} has a unique
  solution $u \in L^1(\RR_+, \CC^n)$.
\end{theorem}

Once global existence of an $L^1$ solution has been obtained, a
natural question is its decay, depending on the decay of the kernel
and data. While many works, for instance on Landau damping, treat this
kind of question, most results fall into one of the following two
types. Either they use weighted spaces and use the stronger assumption
\(\int_0^\infty |K(t)|\, (1+t)^\alpha\, \dd t\) to conclude from \(v\)
in \(O((1+t)^{-\alpha})\) that \(u\) is \(O((1+t)^{-\alpha})\), e.g.\
\cite{dietert-2016-stability-kuramoto,fernandez-gerard-varet-giacomin-2016-landau-kuramoto}.
Or they establish $L^p$ in time estimates without loss for
$p=1,2$. Indeed, while weighted $L^1$ estimates behave well with
respect to convolution formulas such as \eqref{classic_Volterra2},
$L^2$ estimates may be established using \eqref{Volterra_Fourier} and
Plancherel formula, see \cite{albritton-ohm-2022-preprint}.

Note that this is a significant loss of information, and a look at
\eqref{classic_Volterra1} shows that we may expect much more: if
\(K,v\) are \(O((1+t)^{-\alpha})\) for \(\alpha>1\), one may hope that
$u$ is \(O((1+t)^{-\alpha})\) as well because this property is stable
by convolution. Our main result exactly shows this:
\begin{theorem} \label{thm_decay_Volterra}

  Assume that \eqref{spectral} holds. Let $\alpha > 1$, and assume
  that $K$ satisfies
  \begin{align} \label{assumption_exp_K}
    |K(t)| \le C_K (1+t)^{-\alpha}, \quad t \ge 0,
  \end{align}
  for \(\alpha>1\) and a constant \(C_K\).  Then, for any $v$,  the solution $u$ of
  \eqref{classic_Volterra1} satisfies
  \begin{equation*}
  \sup_{t \in \RR_+} (1+t)^{\alpha} |u(t)| \lesssim   \sup_{t \in \RR_+} (1+t)^{\alpha} |v(t)|.
  \end{equation*}
\end{theorem}
\begin{remark}\label{rem:FaouRousset}
  Faou, Horsin and Rousset
  \cite[Corollary~3.3]{faou-horsin-rousset-2021-vlasov-hmf} have a
  related result for specific weights. Our proof is different and
  easily applies for general weights, notably weight
  $\ln(2+t) (1+t)^{-\alpha}$, that will be used later.
\end{remark}

Our proof relies on the analysis of Volterra equations of
non-convolution type carried in
\cite[Chapter~9]{gripenberg-londen-staffans-1990-volterra}. These
equations take the form
\begin{equation} \label{Volterra_non_conv}
  \tilde{u}(t) + \int_J k(t,s)\, \tilde{u}(s)\, \dd s = \tilde{v}(t),
  \quad \text{ for almost every $t$ in $J$,}
\end{equation}
where $J$ is a subinterval of $\RR_+$. An important notion developed
there is the one of \emph{Volterra kernel of type $L^p$}, that we
introduce here only for $p=\infty$.
\begin{definition} \label{defi_kernel_Volterra}

  A Volterra kernel on $J$ is a measurable mapping
  $k : J \times J \rightarrow M_n(\CC)$, such that $k(t,s) = 0$ for
  all $s,t \in J$ with $s > t$.
\end{definition}
\begin{definition} \label{defi_type_Lp}

  A Volterra kernel is said to be of type $L^\infty$ on $J$ if it satisfies
  \begin{equation*}
    \tnorm{k}_{\infty,J} < \infty, \quad \text{ where } \quad
    \tnorm{k}_{\infty,J} := \sup_{t \in J} \int_{J} |k(t,s)|\, \dd s.
  \end{equation*}
\end{definition}

On \(J\) we can define a generalized convolution product
\begin{equation*}
  (k_1 \star k_2)(t,s) := \int_{J} k_1(t,u)\, k_2(u,s) \, \dd u
\end{equation*}
and one can directly verify that the set of Volterra kernels of type
$L^\infty$ on $J$, equipped with $(+,\star)$, is a Banach algebra for
the norm $\tnorm{\cdot}_{\infty,J}$. Moreover, one can show that the
space $L^\infty(J, \CC^n)$ is a left Banach module over it, through
$(k v)(t) = \int_{J} k(t,u)\, v(u)\, \dd u$.

As for classical Volterra equations, one has a notion of resolvent:
\begin{definition} \label{defi_resolvent_Volterra}
  Given a  Volterra kernel $k$ on $J$, a resolvent of $k$ on $J$ is
  another Volterra kernel satisfying
  \begin{equation*}
    r + k \star r = r + r \star k = k.
  \end{equation*}
\end{definition}

As in the convolution case, the resolvent determines the solution.
\begin{lemma}[{\cite[Chapter~9,
    Lemma~3.4]{gripenberg-londen-staffans-1990-volterra}}]
  \label{lem3.4_Volterra}

  If $k$ is a Volterra kernel of type $L^\infty$ on $J$, which has a
  resolvent $r$ of type $L^\infty$ on $J$, then, for any
  $\tilde{v} \in L^\infty(J, \CC^n)$, equation
  \eqref{Volterra_non_conv} has a unique solution
  $\tilde{u} \in L^\infty(J, \CC^n)$, given by
  \begin{equation*}
   \tilde{u}(t) = \tilde{v}(t) - \int_{J} r(t,u)\, \tilde{v}(u)\, \dd u.
  \end{equation*}
  In particular, $\|\tilde{u}\|_{L^\infty(J)} \lesssim \|\tilde{v}\|_{L^\infty(J)}$.
\end{lemma}

Using the standard von Neumann series for perturbations of the
resolvent map in the Banach algebra, we can control the resolvent
around a known resolvent.
\begin{proposition}[{\cite[Chapter~9, Theorem~3.9]{gripenberg-londen-staffans-1990-volterra}}]
  \label{thm3.9_Volterra}

  If $k = k_1 + k_2$ is the sum of two Volterra kernels of type
  $L^\infty$ on $J$, if $k_1$ has a resolvent $r_1$ of type $L^\infty$
  on $J$, and if
  \begin{equation*}
    \tnorm{k_2}_{\infty,J} < \frac{1}{1 + \tnorm{r_1}_{\infty,J}}
  \end{equation*}
  then $k$ has a resolvent of type $L^\infty$ on $J$.
\end{proposition}

\begin{proof}[Proof of \cref{thm_decay_Volterra}]
  For \(\epsilon > 0\) (to be chosen later sufficiently small), we  consider
  \begin{equation*}
    \tilde{u}(t) := (1+\epsilon t)^\alpha u(t), \quad
    \tilde{v}(t) := (1+\epsilon t)^\alpha v(t).
  \end{equation*}
  By \eqref{classic_Volterra1}, they satisfy
  \begin{equation}
    \tilde{u}(t)  + \int_{\RR_+}  k(t,s)\, \tilde{u}(s)\, \dd s = \tilde{v}(t), \quad \forall t \in \RR_+
  \end{equation}
  with the Volterra kernel
  \begin{equation*}
    k(t,s) := \frac{(1+\epsilon t)^\alpha}{(1+\epsilon s)^\alpha}
    K(t-s) \quad
    \text{ if $t > s$},
    \quad k(t,s) := 0 \quad \text{otherwise}.
  \end{equation*}
  By \cref{lem3.4_Volterra} it suffices to show that \(k\) is of type
  \(L^\infty\) on $\RR_+$ and has a resolvent of type $L^\infty$ on
  $\RR_+$.  We decompose $k$ as
  \begin{equation*}
    k(t,s) = K(t-s) 1_{s < t}
    + \left(\frac{(1+\epsilon t)^\alpha}{(1+\epsilon s)^\alpha} - 1\right) K(t-s) 1_{s < t}
    =: k_1(t,s) + k_2(t,s).
  \end{equation*}
  Clearly, $\tnorm{k_1}_{\infty,\RR_+} \le \|K\|_{L^1}$ so that $k_1$
  is of type $L^\infty$. By \cref{thm:existVolterra}, under
  \eqref{spectral}, the convolution Volterra kernel \(K\) has a
  resolvent \(R\), solution of \eqref{classic_Volterra2}. Setting
  $$ r_1(t,s) := R(t-s) \quad \text{ if $t > s$}, \quad r_1(t,s) := 0 \quad \text{otherwise},   $$
  yields the resolvent \(r_1\) of $k_1$ on $\RR_+$. Moreover,
  $\tnorm{r_1}_{\infty,\RR_+} \le \|R\|_{L^1}$ so that $r_1$ is of
  type $L^\infty$. By  \cref{thm3.9_Volterra}, it then suffices to show that  we can
  make \(\tnorm{k_2}_{\infty,\RR_+} \) arbitrarily small by choosing \(\epsilon\) small
  enough. For this last claim,  first consider \(t \le \lambda/\epsilon\) for a
  parameter \(\lambda > 0\). Then
  \begin{equation*}
    \int_0^t |k_2(t,s)|\, \dd s
    \lesssim \int_0^t \left( (1+\lambda)^\alpha - 1 \right) \frac{\dd s}{(1+t-s)^\alpha} \le \left( (1+\lambda)^\alpha - 1 \right) \int_{\RR_+} \frac{\dd s'}{(1+s')^\alpha},
  \end{equation*}
  which can be made arbitrary small by choosing \(\lambda\) small
  enough. This $\lambda$ being fixed,  we consider now \(t > \lambda/\epsilon\) and split the
  integral at \(\delta t\) for \(0<\delta<1\). First,
  \begin{equation*}
  \begin{aligned}
    \int_{\delta t}^t |k_2(t,s)|\, \dd s
  &  \lesssim \int_{\delta t}^t
    \left(
      \left(\frac{1+\epsilon t}{1+\epsilon \delta t}\right)^\alpha - 1
    \right)
    \frac{\dd s}{(1+t-s)^\alpha} \\
 &    \lesssim  \left(  \left(\frac{1+\epsilon t}{1+\epsilon \delta t}\right)^\alpha - 1
    \right) \int_{\RR_+} \frac{\dd s'}{(1+s')^\alpha}  \lesssim \sup_{t' \in \RR_+}  \left(\frac{1+t'}{1+\delta t'}\right)^\alpha - 1,
    \end{aligned}
  \end{equation*}
  which can be made arbitrary small by choosing some \(\delta<1\)
  close to \(1\). Then, for fixed \(\lambda,\delta\) we find
  \begin{align*}
    \int_0^{\delta t} |k_2(t,s)|\, \dd s
   & \le \sup_{s \in (0,\delta t)} |K(t-s)|
    \int_0^{\delta t}
    \frac{(1+\epsilon t)^\alpha}{(1+\epsilon s)^\alpha}
    \dd s    \lesssim \frac{(1+\epsilon t)^\alpha}{(1+t)^\alpha} \int_{\RR_+} \frac{\dd s}{(1+\eps s)^\alpha} \\
    & \lesssim   \frac{(1+\epsilon t)^\alpha}{\eps (1+t)^\alpha} \lesssim \sup_{\tau > \frac{\lambda}{\eps}}  \frac{(1+\epsilon \tau)^\alpha}{\eps \tau^\alpha}
     \lesssim \left(\lambda^{-1}+1\right)^\alpha \eps^{\alpha-1},
  \end{align*}
  which again can be made arbitrary small by choosing \(\epsilon\)
  sufficiently small.
\end{proof}

\section{Isotropic suspensions in the inviscid case}
\label{sec:inviscid}
As the incoherent state is rotational invariant, we can always choose
a coordinate system so that \(k \in \bbS^2\) equals the coordinate
vector \(e := (0,0,1)^t\). More precisely, if $\psi_k[\psi^{in}]$ is
the solution of \eqref{LSS-mode} with initial data $\psi^{in}$, one
has for any rotation matrix $R$:
$$ \psi_{Rk}[\psi^{in} (R^{-1} \cdot )](t,p) = \psi_k[\psi^{in}](t,R^{-1} p), $$
so that it is enough to assume $k = e$ to prove \cref{thm1} or
\cref{thm2}. For easier readability we then also drop the explicit
dependence in the index \(k\). The system \eqref{LSS-mode} reduces to
\begin{equation} \label{SS-mode-reduced}
  \partial_t \psi = L_1 \psi + \bar L_2 \cdot  u[\psi]
\end{equation}
where $L_1 := L_{1,e}$, $u[\psi] = u_e[\psi]$,
$\bar L_2 = \bar L_{2,e}$, see definitions
\eqref{def_L1}-\eqref{def_uk}-\eqref{def_L2}. By standard methods, for
any $\psi^{in} \in H^s(\bbS^2)$, $s \ge 0$, there exists a unique
solution
$\psi \in C(\RR_+, H^s(\bbS^2)) \cap C^1(\RR_+, H^{s-1}(\bbS^2))$ if
$\nu = 0$, resp.
$\psi \in C(\RR_+, H^s(\bbS^2)) \cap L^2_{loc}(\RR_+,
H^{s+1}(\bbS^2))$ if $\nu > 0$, of \eqref{SS-mode-reduced}. The point
is to obtain the decay estimates, first for $u$, then for $\psi$
itself. We consider in this section the case $\nu = 0$.

\subsection{Volterra equation on $u$}

As indicated in \cref{sec:intro:key-ideas}, we can rewrite the
evolution for $u$ as a Volterra equation
$$ u(t) + K \star u(t) = U(t) $$
with
 \begin{equation} \label{UK}
  U(t) := u[\e^{tL_1} \psi_e^{in}], \quad   K(t) v := - u[\e^{tL_1} (\bar L_2 \cdot v)], \quad v \in \CC^3.
\end{equation}
As $\nu=0$, by the definition of \(u[\cdot]\), we find explicitly for
a test function \(\phi\) that
\begin{equation}
  \label{eq:inviscid-u}
  u[\e^{tL_1} \phi]
  = i \eps \int_{\bbS^2} \bbP_{e^\perp} p\, (e \cdot p)
  (\e^{tL_1} \phi)(p)\, \dd p
  = i \eps \int_{\bbS^2} \bbP_{e^\perp} p\, (e \cdot p)
  \phi(p)\, \e^{-ie \cdot pt} \dd p.
\end{equation}
We recognize a Fourier transform over the sphere, with well-known decay properties, quantified in
\begin{lemma} \label{lemma_decay_Fourier}

  Let \(M \in \NN\), \(\delta > 0\) and $F\in H^{2M+1+\delta}(\bbS^2)$
  be a function over $\bbS^2$. For a unit vector \(e \in \bbS^2\)
  define the integral
  \begin{equation*}
    I(t) = \int_{\bbS^2} \e^{i  \, e \cdot p \, t} F(p) \, \dd p, \quad t \ge 0.
  \end{equation*}
  Then there exist complex numbers $c_{m,\pm}$, $1 \le m \le M$, and
  $I_M= I_M(t)$ such that
  \begin{align}
    & c_{m,\pm} \in
      \operatorname{span}\Big(\big\{ \partial^\beta F(\pm e), \quad |\beta| \le 2m-2\big\}\Big) \\
    & \forall t \ge 0, \quad |I_M(t)| \le \frac{C}{(1+ t)^{M+1}}  \|F\|_{H^{2M+1+\delta}(\bbS^2)}
  \end{align}
  and
  \begin{align}
    I(t) = \sum_{m=1}^{M} \big( c_{m,+} \e^{i t} + c_{m,-} \e^{-it} \big) (1+t)^{-m} + I_M(t).
  \end{align}
\end{lemma}
For completeness, we shall provide the proof of this lemma, directly
inspired from the lecture notes~\cite{Choi}, at the end of this
subsection. In the special case of our source $U$ and kernel $K$, the
integrand $F$ of the lemma contains the projection \(\bbP_{k^\perp}\),
that is vanishing at \(p = \pm k\) so that \cref{lemma_decay_Fourier}
with \(M=1\) implies the bounds
\begin{align*}
  |U(t)| &\lesssim (1+t)^{-2} \| \psi^{in} \|_{H^{3+\delta}}, \\
  |K(t)| &\lesssim (1+t)^{-2}.
\end{align*}
Hence we find by \cref{thm_decay_Volterra} that under the spectral
condition
\begin{equation}
  \label{eq:spectral-stab}
 \det (I + \cL K(\lambda)) \neq 0 \quad \text{ for all } \: \Re \lambda \ge 0,
\end{equation}
the first bound \eqref{eq:inviscid-u-decay} on the decay of \(u\)
holds. To complete the proof of \cref{thm1}, it remains to
check when assumption \eqref{eq:spectral-stab} is satisfied, and
finally to analyse the decay of $\psi$ itself. This will be done in
the following Subsections \ref{sec:inviscid:stability-condition} and
\ref{sec:inviscid:control-psi}.

\begin{proof}[Proof of \cref{lemma_decay_Fourier}]

  We introduce a smooth partition of unity $\chi_0$, $\chi_\pm$, with
  $\chi_0 + \chi_+ + \chi_- = 1$ over $\bbS^2$, where $\chi_0$ is
  supported away from $p = \pm e$ and $\chi_\pm$ is supported in a
  neighborhood of $\pm e$. We decompose
  \begin{equation*}
    I(t) =  I_0(t) + I_+(t) + I_-(t),
    \quad I_l(t) = \int_{\bbS^2} F(p) \chi_l(p)  \, \dd p,
    \quad l \in \{0,+,-\}.
  \end{equation*}

  By the rotational symmetry, one can introduce coordinates such that
  \(e\) is along the \(z\)-axis. Then any point \(p \in \bbS^2\) can
  be parametrized as
  \begin{equation}
    \label{eq:gamma-z-param}
    p = p(\gamma,z)
    = (\sqrt{1-z^2}\, \sin \gamma, \sqrt{1-z^2}\, \cos \gamma, z)^T
  \end{equation}
  for \(z \in [-1,1]\) and \(\gamma \in [0,2\pi)\). The surface
  measure on the sphere is $\dd \sigma = \dd \gamma\, \dd z$ so that
  \begin{align}
    I_0(t) = \int_{-1}^1 \e^{-izt}  F_0(z) \,\dd z, \quad F_0(z)
    = \int_{0}^{2\pi} (F \chi_0)(p(\gamma,z)) \, \dd \gamma.
  \end{align}
  The key point is that $F_0$ is compactly supported in $(-1,1)$ so that
  its extension by zero to $\RR$, still denoted $F_0$, is in
  $H^{2M+2}_c(\RR)$. We find in particular that
  \begin{align}
    \forall \,  t\ge 1, \quad |t^{M+1} I_0(t)|
    = |t^{M+1} \hat{F_0}(t)| \le \|F_0\|_{W^{M+1,1}(\RR)} \le
    \|F_0\|_{W^{M+1,1}(\bbS^2)}
    \lesssim \|F\|_{H^{2M+1+\delta}(\bbS^2)}.
  \end{align}

  It remains to treat $I_+$  ($I_-$ can be handled in the same way).  Let
  \begin{equation*}
    \varphi : B(0,1) \subset \RR^2 \rightarrow \bbS^2, \quad x \rightarrow (x,\sqrt{1-|x|^2}).
  \end{equation*}
  We find
  \begin{equation*}
    I_+(t) = \int_{B(0,1)} \e^{i \sqrt{1-|x|^2} t}
    \left( F_0 \, \chi_+  \, j  \right)(\varphi(x)) \, \dd  x
  \end{equation*}
  where $j$ is the Jacobian from the change of variable.

  The key point is that the phase $\sqrt{1-|x|^2}$ has a non-degenerate
  critical point at $x = 0$, with Hessian matrix at zero being
  $-2 I$. By Morse lemma, there exists a smooth diffeomorphism $\psi$
  from a neighborhood $U$ of $0$ in $\RR^2$ to $B(0,\eta)$, for $\eta$
  small enough, so that $\sqrt{1-|\psi(y)|^2} = 1-|y|^2$. By taking
  the support of \(\chi_+\) sufficiently small, we can perform another
  change of variables and arrive at
  \begin{equation*}
    I_+(t) = \e^{i t}\int_{U} \e^{- i |y|^2 t}  F_+(y) \, \dd y
  \end{equation*}
  where $F_+$ is the product of $F_0 \circ \varphi \circ \psi$ with a
  smooth function, compactly supported in $U$ coming from the
  Jacobian. Extending $F_+$ by zero outside $U$, we end up with
  \begin{equation*}
    I_+(t) = \e^{i t} \int_{\RR^2}  \e^{- i |y|^2 t}  F_+(y) \, \dd y
    = \frac{\e^{it-i\frac{\pi}{4}}}{t} \int_{\RR^2} \e^{- i \frac{|\xi|^2}{t}} \hat{F}_+(\xi)\, \dd\xi
  \end{equation*}
  with the Fourier transform \(\hat{F}_+\) of \(F_+\) and where the last
  line comes from Plancherel identity.

  We can then perform a Taylor expansion
  \begin{equation*}
    \e^{- i \frac{|\xi|^2}{t}}
    = \sum_{m=1}^M \frac{1}{(m-1)!}\left(- i \frac{|\xi|^2}{t}\right)^{m-1}  + O\Big( \frac{|\xi|^{2M}}{t^M} \Big).
  \end{equation*}
  Setting
  \begin{equation*}
    c_{m,+} = \e^{-i\frac{\pi}{4}} \frac{1}{(m-1)!} (- i)^{m-1}
    \int_{\RR^2} |\xi|^{2m-2} \hat{F}_+(\xi)\, \dd \xi
    =  2\pi \e^{-i\frac{\pi}{4}} \frac{1}{(m-1)!} (- i)^{m-1} [(-\Delta)^{m-1} F_+](0)
  \end{equation*}
  we obtain that
  \begin{equation*}
    I_+(t) = \sum_{m=1}^M \frac{c_{m,+} \e^{it}}{t^m} + I_{M,+}(t),
    \quad |I_{M,+}(t)| \le \frac{C}{t^{M+1}} \| |\xi|^{2M} \hat{F}_+ \|_{L^1}.
  \end{equation*}
  One can then notice that
  \begin{equation*}
    \begin{aligned}
      \| |\xi|^{2M} \hat{F}_+ \|_{L^1(\RR^2)}
      &\le \| (1{+}|\xi|)^{-1-\delta}\|_{L^2(\RR^2)}
        \| (1{+}|\xi|)^{2M+1+\delta} \hat{F}_+ \|_{L^2(\RR^2)} \\
      &\lesssim \|  F_+ \|_{H^{2M+1+\delta}(\RR^2)} \lesssim \|F\|_{H^{2M+1+\delta}(\bbS^2)}.
    \end{aligned}
  \end{equation*}
  The result of the lemma follows directly.
\end{proof}

\subsection{Stability condition}
\label{sec:inviscid:stability-condition}

We now study the stability condition \eqref{eq:spectral-stab}. This
condition was already studied through explicit numerical computations
in \cite{saintillan-shelley-2008-instabilities}, but here we provide
another angle through the argument principle which allows a complete
solution.

We first compute the determinant.
\begin{lemma}
  \label{thm:determinant}
  For \(\lambda \in \CC\) with \(\Re \lambda \ge 0\) we have
  \begin{align}
    \label{eq:spectral-boundary-outside}
    \det(1 + \cL K(\lambda))
    &= \left(
      1 + \frac{3\Gamma \eps}{4}
      \int_{-1}^{+1} \frac{z^2(1-z^2)}{\lambda+iz}\, \dd z
      \right)^2
    &
    & \text{if } \lambda \not \in i(-1,1).\\
    \det(1 + \cL K(\lambda))
    &= \left(
      1 +
      \frac{3\Gamma\eps\pi}{4} b^2(1{-}b^2)
      + \operatorname{PV}
      \frac{3\Gamma \eps}{4i}
      \int_{-1}^{+1} \frac{z^2(1{-}z^2)}{b+z}\, \dd z
      \right)^2
    &
    & \text{if } \lambda = i b \text{ for } b \in (-1,1).
  \end{align}
\end{lemma}
\begin{proof}
  We remind that
  \begin{equation*}
    \bar L_2 = \frac{3\Gamma i}{4\pi} (p\cdot e) p
  \end{equation*}
  We again take the parametrization \eqref{eq:gamma-z-param}. We then
  deduce from \eqref{UK} and \eqref{eq:inviscid-u} for
  \(\Re \lambda > 0\) that
  \begin{equation*}
    \cL K(\lambda)
    = \frac{3\Gamma\eps}{4\pi}
    \int_{z=-1}^{1} \int_{\gamma=0}^{2\pi}
    \frac{z^2(1-z^2)}{\lambda + iz}
    \begin{pmatrix}
      \sin^2 \gamma & \sin \gamma \cos \gamma & 0 \\
      \sin \gamma \cos \gamma & \cos^2 \gamma & 0 \\
      0 & 0 & 0
    \end{pmatrix}
    \dd \gamma\, \dd z.
  \end{equation*}

  This directly yields the expression for \(\Re \lambda > 0\). For
  \(\Re \lambda =0\) the result follows from continuity and Plemelj
  formula.
\end{proof}

\begin{remark} \label{rem_PV}

  The principal value can be computed explicitly: for all
  $b \in (-1,1)$,
  \begin{align} \label{exp_PV}
    \operatorname{PV} \int_{-1}^{1} \frac{z^2 (1-z^2)}{b + z} \dd z
    = 2 b^3 - \frac43 b + (b^4 -  b^2) \ln \frac{1-b}{1+b}.
  \end{align}
  This is an odd function, that vanishes at $0$ and at $\pm b_c$ where
  $b_c > 0$ can be evaluated numerically to
  $$b_c \approx 0.62375.$$
  Moreover, the function is negative on $(0,b_c)$, positive on
  $(b_c,1)$.
\end{remark}

\begin{proposition}\label{thm:spectral-condition}

  For the inviscid system, one has depending on \(\eps\)
  \begin{enumerate}
  \item If \(\eps = 1\), then the spectral condition
    \eqref{eq:spectral-stab} is satisfied for all \(\Gamma \in \RR_+\).
  \item If \(\eps = -1\), the spectral
    condition is satisfied if and only if \(\Gamma < \Gamma_c\), where
    \begin{equation*}
      \frac{1}{\Gamma_c} = \frac{3\pi}{4} b_c^2(1-b_c^2).
    \end{equation*}
  \end{enumerate}
\end{proposition}
\begin{proof}
  For \(\Re \lambda \ge 0\) define the analytic function \(F(\lambda)\) by
  \begin{equation*}
    \begin{aligned}
      F(a+ib)
      &= \frac{3i}{4}
        \int_{z=-1}^{+1} \frac{z^2(1-z^2)}{b+z}\, \dd z
      & &\text{if } (a,b) \not \in \{0\} \times (-1,1), \\
      F(ib)
      &=
        - \frac{3\pi}{4} b^2 (1-b^2) +
        \frac{3i}{4}
        \operatorname{PV} \int_{z=-1}^{+1} \frac{z^2(1-z^2)}{b+z}\, \dd z
      & &\text{if } b \in (-1,1).
    \end{aligned}
  \end{equation*}

  By \cref{thm:determinant}, there exists then an eigenmode if and
  only if \(F\) attains in the right half plane the value~\(\eps/\Gamma\).

  By the explicit expression, we also see that \(F(\lambda) \to 0\) as
  \(|\lambda| \to \infty\). Hence by the argument principle the
  attained values are exactly those values encircled by the curve
  \(b \mapsto F(ib)\). This curve is plotted in
  \cref{fig:curve-spectral} but can also be understood analytically.

  \begin{figure}[tb]
    \centering
    \begin{tikzpicture}
      \begin{axis}[
        xlabel={\(\Re F\)},
        ylabel={\(\Im F\)}
        ]
        \addplot [very thick,green,fill=green!20!white]
        table {curve.txt};
      \end{axis}
    \end{tikzpicture}
    \caption{Encircled area of the curve \(b \mapsto F(b)\) for the
      spectral condition in \cref{thm:spectral-condition}.}
    \label{fig:curve-spectral}
  \end{figure}
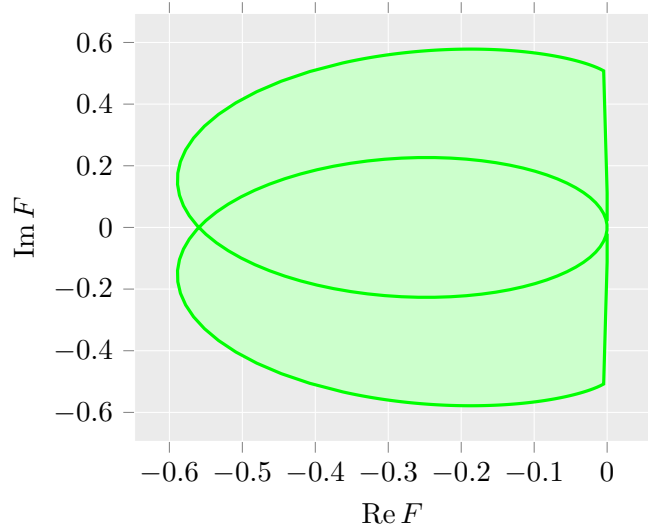

  For \(|b| \ge 1\), the expression directly shows that \(F(b)\) is
  purely imaginary, while for \(|b| \le 1\) we find that
  \(\Re F \le 0\). Hence it cannot encircle positive real numbers,
  which proves the first statement of the proposition.

  In \( \{|b| < 1\}\) the curve crosses the real axis at \(b=0\) and
  \(b=\pm b_c\). By the expression of \(F\), we see that it indeed
  crosses the real axis at \(-3\pi b_c^2(1-b_c^2)/4\). Hence we then
  have an eigenmode if and only if
  \begin{equation*}
    \frac{1}{\Gamma} \le \frac{3\pi}{4} b_c^2(1-b_c^2),
  \end{equation*}
  which shows the claimed stability.
\end{proof}

\subsection{Control on $\psi$} \label{subsec:control_psi}
\label{sec:inviscid:control-psi}

With the control of \(u\) we go back to prove
\eqref{eq:inviscid-psi-decay}. Again, it is enough to treat the case
$k = e$. Let $\varphi \in H^{1+\delta}(\bbS^2)$. From
\eqref{eq:duhamel-psi}, we deduce
\begin{equation*}
  \int_{\bbS^2} \psi(t,p) \varphi(p)\, \dd p
  = \int_{\bbS^2} \e^{i t e \cdot p}  \psi^{in}(p)  \varphi(p) \, \dd
  p
  +
  \int_0^t \int_{\bbS^2} \e^{i (t-s) e \cdot p} \varphi(p) \bar L_2(p)
  \cdot u(s) \, \dd p\, \dd s =: I_1(t) + I_2(t).
\end{equation*}
We apply \cref{lemma_decay_Fourier} with $M=0$, to obtain
\begin{equation*}
  |I_1(t)| \le \frac{C}{(1+t)} \| \psi^{in}  \varphi \|_{H^{1+\delta}}
  \lesssim \frac{1}{1+t}  \| \psi^{in} \|_{H^{1+\delta}}  \| \varphi \|_{H^{1+\delta}},
\end{equation*}
as $H^{1+\delta}(\bbS^2)$ is an algebra. As regards $I_2$, we write
\begin{equation*}
  \begin{aligned}
    I_2(t)
    &= \int_0^{t/2} \Big( \int_{\bbS^2} \e^{i (t-s) e \cdot p}
      \varphi(p)  \bar L_2(p) \dd p \Big) \cdot u(s)\, \dd s + \int_{t/2}^t
      \int_{\bbS^2} \e^{i (t-s) e \cdot p} \varphi(p)   \bar L_2(p) \cdot
      u(s) \, \dd p\, \dd s \\
    &=: I_{2,1}(t) + I_{2,2}(t).
  \end{aligned}
\end{equation*}
We bound the parenthesis in the first term using again
\cref{lemma_decay_Fourier} with $M=0$: it follows
\begin{align*}
  |I_{2,1}(t)| & \lesssim  \int_0^{t/2}  \frac{1}{1+(t-s)} \|  \varphi
                 \bar L_2 \|_{H^{1+\delta}} |u(s)|\, \dd s \\
               & \lesssim  \frac{1}{1+t}  \|\varphi\|_{H^{1+\delta}}
                 \int_0^{t/2}  \frac{1}{(1+s)^2}\, \dd s\, \|\psi^{in}\|_{H^{3+\delta}} \\
               &  \lesssim  \frac{1}{1+t}  \|\varphi\|_{H^{1+\delta}}   \|\psi^{in}\|_{H^{3+\delta}},
\end{align*}
where we have used \eqref{eq:inviscid-u-decay} for the second
inequality. Eventually, we use again \eqref{eq:inviscid-u-decay} to
get
\begin{align*}
  |I_{2,2}(t)| & \lesssim  \int_{t/2}^t   \frac{1}{(1+s)^2}\, \dd s\,
                 \|\psi^{in}\|_{H^{3+\delta}} \| \varphi \bar L_2\|_{L^\infty} \\
               & \lesssim \frac{1}{1+t}  \|\varphi\|_{H^{1+\delta}} \|\psi^{in}\|_{H^{3+\delta}},
\end{align*}
where we applied the Sobolev imbedding
$H^{1+\delta}(\bbS^2) \hookrightarrow L^\infty(\bbS^2)$. This
concludes the proof of \eqref{eq:inviscid-psi-decay} and of
\cref{thm1}.

\section{Decay estimates for advection-diffusion on the sphere}\label{sec:diffusive}

In this central section of our paper we prepare for the results under
small angular diffusion \(\nu\) by proving \cref{prop:hypocoercf} and
\cref{thm:vector-field-decay}.

Here we study mixing estimates for the semigroup $\e^{tL_{1,k}}$ where
$L_{1,k} = - i k \cdot p + \nu \Delta$, $k \in \bbS^2$. As explained
at the beginning of \cref{sec:inviscid}, there is no loss of generality in
assuming \(k = e = (0,0,1)\). We shall later use spherical coordinates
$(\theta,\varphi)$, with colatitude $\theta \in (0, \pi)$ and longitude
$\varphi \in (0,2\pi)$, so that
$p = \sin \theta \cos \varphi\, e_x + \sin \theta \sin \varphi\, e_y +
\cos\theta\, e$.

Setting \(\psi(t) = \e^{t L_{1,k}} \psi^{in}\), we thus study decay
estimates for solutions to
\begin{align} \label{eq:ADEfourier}
 \partial_t \psi+ i p \cdot e\, \psi = \nu \Delta \psi , \qquad p\in \bbS^2,
\end{align}
where, from now on, we suppress the subscript on all differential
operators, as it is understood that they are all with respect to the
variable $p$ on the sphere \(\bbS^2\).

A first feature of \eqref{eq:ADEfourier} is that the interaction of
transport and diffusion leads to an enhanced diffusion time scale
$O(1/\sqrt\nu)$ that is much shorter than the heat equation one
$O(1/\nu)$. This phenomenon is analysed in details in
\cref{sub:enhancedPS}. Furthermore, we exhibit an auxiliary time
$\nu^{-1/3}$ for the decay of $\|\nabla (p \cdot e) \psi\|_{L^2}$, a
quantity that vanishes near the pole. This additional time scale is
coherent with classical results for enhanced dissipation in the
Euclidean setting, where the absence of critical points also leads to
typical time $\nu^{-1/3}$.

In the following \cref{sub:enhancedVF}, we build a vector field
$J_\nu$ under the form
$$ J_\nu  = \alpha_\nu(t) \nabla   + i \beta_\nu(t) \nabla (p \cdot e) $$
such that, roughly:
\begin{enumerate}[(i)]
\item $\alpha_\nu \sim 1$,  $\beta_\nu \sim t$ over times $\nu^{-1/2}$;
\item $J_\nu \psi$ is well-controlled over times $\nu^{-1/2}$.
\end{enumerate}
In the usual Euclidean setting, where $e \cdot p$ is replaced by
$e \cdot v$, $v \in \RR^3$, the vector field $J = \nabla_v + i t e$ is
an easy and convenient choice: it commutes with the advection-diffusion
operator, so that $J \psi$ can be controlled easily for all times.  In
the case of the sphere, we do not know how to construct such a
commuting vector field. Designing a $J_\nu$ for which we can show
properties (i) and (ii) is difficult, and relies on the refined
estimates of \cref{sub:enhancedPS}.

By these vector fields we can obtain the first part of the mixing
estimates of \cref{thm:vector-field-decay} up to time \(\nu^{-1/2}\)
in \cref{sub:mixing}. For the estimates on the longer time scale, we
adapt the estimates in \cref{sub:mixing-long} which then concludes the
proof of \cref{thm:vector-field-decay}.

\subsection{Hypocoercive estimate for advection-diffusion}
\label{sub:enhancedPS}

In this subsection we will introduce the hypocoercive estimates which
will yield the enhanced dissipation of \cref{prop:hypocoercf}.

Using the covariant derivatives \(\nabla\) as discussed in
\cref{app:covariant}, we find in our case of \(\bbS^2\) that
$$ \nabla \Delta \psi =  \Delta \nabla \psi - \nabla \psi $$
where the Laplacian $\Delta = \tr(\nabla^2)$ is the connection
Laplacian (the correction comes from the Ricci curvature tensor which
equals the metric on the sphere). Taking the covariant derivative of
\eqref{eq:ADEfourier}, we therefore get
\begin{align}\label{eq:ADEgrad}
  \partial_t \nabla \psi  +  i p \cdot e  \nabla \psi
  + i \nabla(p\cdot e) \psi
  = \nu \Delta \nabla \psi  - \nu \nabla \psi.
\end{align}
We explicitly compute \(\Delta(\nabla(p\cdot e) \psi)\) in
\cref{thm:laplace-mixed-term} and this concrete expression yields
\begin{align}\label{eq:ADEfac}
  \partial_t (\nabla(p\cdot e) \psi)  +  i p \cdot e (\nabla(p\cdot e)
  \psi)
  = \nu \Delta (\nabla(p\cdot e) \psi) +  \nu \nabla(p\cdot e) \psi
  + 2  \nu (p \cdot e) \nabla  \psi.
\end{align}
We further note that \(|p\cdot e| \le 1\) and
\(|\nabla(p\cdot e)| \le 1\). In spherical coordinates, we find explicitly that
\(|\nabla(p\cdot e)| = \sin \theta\) which is uniformly lower-bounded
away from the poles where it vanishes linearly.

We now derive several energy identities for the solution to
\eqref{eq:ADEfourier}. In all what follows, $\l\cdot,\cdot\r$ and
$\|\cdot\|$ stand for the complex scalar product (with the conjugate
on the second variable) and norm on $L^2(\bbS^2)$ respectively. For
simplicity, we will use the same notations for the scalar product and
norm of $L^2$ vector fields, or more generally for $L^2$ sections of
any tensor bundle over $\bbS^2$.  A direct computation using the
antisymmetry of the operator $ ip \cdot e$ allows us to derive the
$L^2$ balance
\begin{equation}
  \label{eq:L2balance}
  \frac12\ddt \|\psi\|^2 +\nu\|\nabla \psi\|^2 = 0.
\end{equation}
By testing \eqref{eq:ADEgrad} with $\nabla \psi$, we obtain the $H^1$
balance
\begin{equation}
  \label{eq:H1balance}
  \frac12\ddt \|\nabla \psi\|^2+\nu\|\nabla \nabla \psi\|^2
  + \nu \|\nabla \psi\|^2
  = -\Re \l i \nabla(p\cdot e)\psi, \nabla \psi \r
  \le \| \nabla(p\cdot e) \psi \| \, \| \nabla \psi \|.
\end{equation}
Next we find
\begin{equation}
  \label{eq:crossterm}
  \begin{aligned}
    \ddt \Re \l i \nabla(p\cdot e) \psi,
    \nabla \psi \r
    + \| \nabla(p\cdot e) \psi \|^2
    & = \nu   \Re\l i \Delta ( \nabla(p\cdot e) \psi ) ,  \nabla \psi \r   + \nu \l i \nabla (p \cdot e) \psi , \nabla \psi \r \\
    &\quad +  2 \nu   \Re\l i (p\cdot e) \nabla \psi,  \nabla \psi \r + \nu \Re\l i  \nabla(p\cdot e) \psi  , \nabla \Delta \psi \r
    \\
    & = - \nu  \Re\l i \nabla ( \nabla (p\cdot e) \psi )  ,  \nabla^2  \psi \r   -  \nu  \Re\l i (p\cdot e) \psi  , \Delta \psi \r \\
    &\quad -  \nu \Re\l  i \nabla \cdot \big( \nabla(p\cdot e) \psi  \big) ,  \Delta \psi \r \\
    &\le
      4 \nu \| \nabla^2 \psi \|\,
      \big(\| \psi \| + \| \nabla (\nabla(p\cdot e) \psi) \|\big).
  \end{aligned}
\end{equation}
From \eqref{eq:ADEfac} we find
\begin{equation}
  \label{eq:gamma-weight}
  \begin{aligned}
    &\frac 12 \ddt \| \nabla(p\cdot e) \psi \|^2
    + \nu \| \nabla(\nabla(p\cdot e) \psi) \|^2
    + \nu \|\nabla(p\cdot e) \psi \|^2 \\
    &\qquad\qquad= 2 \nu \Re
      \l \nabla(p\cdot e) \psi + (p\cdot e) \nabla \psi,
      \nabla(p \cdot e) \psi \r \\
    &\qquad\qquad= 2 \nu \|\nabla(p\cdot e) \psi \|^2 + \nu \int (p\cdot e)
      \nabla(p\cdot e) \cdot \nabla(|\psi|^2) \\
    &\qquad\qquad\le 2 \nu \|\nabla(p\cdot e) \psi \|^2
      + 2 \nu \| \psi \|^2.
  \end{aligned}
\end{equation}
For positive constants $a,b,c$ to be chosen later \emph{independently}
of $\nu$, define the energy functional
\begin{equation}\label{eq:hypofun:time}
  E[\psi] =
  \frac12 \Big[\|\psi\|^2+a \nu t \|\nabla \psi\|^2
  +2b \nu t^2  \Re \l i \nabla(p\cdot e) \psi  ,\nabla \psi \r
  + c \nu t^3 \| \nabla(p \cdot e) \psi \|^2\Big].
\end{equation}
Such a form of time-dependent functional takes inspiration from \cite{LZ21,DelZotto21,WZ19}.
Assuming that \(2b^2 \le ac\), the mixed term can be controlled by the
squares so that
\begin{align}\label{eq:coercivity:time}
  E[\psi] \ge \frac12
  \left[
  \|\psi\|^2
  + \frac{a \nu t}{2} \|\nabla \psi\|^2
  + \frac{c \nu t^3}{2} \| \nabla(p \cdot e) \psi \|^2
  \right].
\end{align}
Moreover, from \eqref{eq:L2balance}--\eqref{eq:gamma-weight}, we deduce
that $E[\psi]$ satisfies the identity
\begin{equation}\label{eq:energy:time}
  \begin{aligned}
    \ddt E[\psi]
    &\le - \left(\nu - \frac{a\nu}{2}\right) \| \nabla \psi \|^2
      - a\nu^2 t \| \nabla^2 \psi \|^2
      - \left(b\nu t^2 - \frac{3c\nu t^2}{2}\right)
      \| \nabla(p\cdot e) \psi \|^2\\
    &\quad
      - c \nu^2 t^3 \| \nabla(\nabla(p\cdot e) \psi) \|^2
      + 2 b \nu t \| \nabla \psi \|\, \|\nabla(p\cdot e) \psi \|
      + a \nu t \| \nabla(p\cdot e) \psi \|\, \| \nabla \psi \|\\
    &\quad
      + 4 b \nu^2 t^2 \| \nabla^2 \psi \|\,
      \big( \| \psi \| + \| \nabla (\nabla(p\cdot e) \psi) \| \big)
      + 2 c \nu^2 t^3
      \big( \| \nabla(p\cdot e) \psi \|^2 + \| \psi \|^2 \big).
  \end{aligned}
\end{equation}
It follows that
\begin{equation}\label{eq:energy:time-split}
  \begin{aligned}
    \ddt E[\psi]
    &\le - \left(\frac{3\nu}{4} - \frac{a\nu}{2}\right) \| \nabla \psi \|^2
      -  \frac{a\nu^2 t}{2} \| \nabla \nabla \psi \|^2\\
    &\quad
      - \left(b\nu t^2 - \frac{3c\nu t^2}{2}
      - 2 c \nu^2 t^3 - 2a^2 \nu t^2 - 8b^2 \nu t^2
      \right)
      \| \nabla(p\cdot e) \psi \|^2\\
    &\quad
      - \left(c \nu^2 t^3 - \frac{8b^2 \nu^2 t^3}{a}
      \right)
      \| \nabla(\nabla(p\cdot e) \psi) \|^2
      + \left(\frac{16b^2}{a}+2c\right) \nu^2 t^3 \| \psi \|^2.
  \end{aligned}
\end{equation}
We will choose the constants according to the following lemma, which is
crucial for the subsequent analysis and holds for the longer
time-scale \(\nu^{-1}\).
\begin{lemma} \label{thm:hypo-constants}
  For 	any \(\delta > 0\), there exists \(b_0\) such that for the constants
  \(a=b^{2/3}\), \(0<b<b_0\) and \(c= 32b^2/a\) and for all times
  \(t \le \delta \nu^{-1}\) it holds for \(0<\nu\le 1\) that
  \begin{equation*}
    \ddt E[\psi]
    + \frac{\nu}{2} \| \nabla \psi \|^2
    +  \frac{a\nu^2 t}{2} \| \nabla \nabla \psi \|^2
    + \frac{b\nu t^2}{2}
    \| \nabla(p\cdot e) \psi \|^2
    + \frac{c \nu^2 t^3}{2}
    \| \nabla(\nabla(p\cdot e) \psi) \|^2
    \le \left(\frac{16b^2}{a}+2c\right) \nu^2 t^3 \| \psi \|^2.
  \end{equation*}
\end{lemma}
\begin{proof}
  By the choices \(a=b^{2/3}\), \(c=32b^2/a\) and the constraint
  \(t \le \delta \nu^{-1/2}\) the result follows from
  \eqref{eq:energy:time-split} as soon as
  \begin{equation*}
    \frac{3c}{2}
    + 2 c \delta
    + 2 a^2
    + 8 b^2
    \le \frac{b}{2}
    \quad \Leftrightarrow \quad
    \left(\frac 32 + 2\delta\right) 16 b^{4/3} + 2 b^{4/3}
    + 8 b^2
    \le \frac b2.
  \end{equation*}
  Hence we find the result for a small enough \(b_0\).
\end{proof}

To conclude the enhanced dissipation from the previous lemma, we need
to control \(\| \psi \|\). This is achieved by an interpolation
inequality involving \(\| \nabla(p\cdot e) \psi\|\), which gives a
good control apart from the poles \(\pm e\), and
\(\| \nabla \psi \|\), see also
\cite[Lemma~4.3]{albritton-ohm-2022-preprint}.

\begin{lemma}\label{thm:hypo-interpolation}
  For all $\sigma\in (0,1]$, all vectors $e \in \bbS^2$ and all
  complex-valued $g\in H^1(\bbS^2)$, the following inequality holds
  \begin{equation}\label{eq:spectralgap}
    \sigma^{1/2}\|g\|^2\le \frac{\sigma}{2} \| \nabla g\|^2
    + 2 \| \nabla(p\cdot e) \,g  \|^2.
  \end{equation}
\end{lemma}

\begin{proof}
  Introducing the spherical coordinates $(\theta, \varphi)$, with
  latitude $\theta \in (0, \pi)$ and longitude $\varphi \in (0,2\pi)$,
  the inequality \eqref{eq:spectralgap} becomes
  \begin{align}\label{eq:spectralgap:proof}
    \sigma^{1/2}\|g\|^2
    \le \frac{\sigma}{2} \int_{\bbS^2}
    \left( |\partial_\theta g|^2 + \frac{1}{\sin^2\theta }
    |\partial_\varphi g|^2\right) \sin \theta\, \dd\theta\, \dd\varphi
    + 2 \|g\sin\theta \|^2.
  \end{align}
  Now, for  $\sigma\leq 1$, we have
  \begin{align}\label{eq:spectral1}
    \sigma^{1/2}\|g\|^2 = \sigma^{1/2}\|g \sin\theta \|^2+ \sigma^{1/2}\|g \cos\theta \|^2\leq
    \| g\sin\theta \|^2+ \sigma^{1/2}\| g\cos\theta \|^2.
  \end{align}
  Moreover, an integration by parts entails
  \begin{align*}
    \| g\cos\theta \|^2
    &= \int_{\bbS^2} \cos\theta\partial_\theta (\sin\theta) |g|^2\sin\theta
      \,\dd\theta\,\dd\varphi
      =\| g\sin\theta   \|^2-\| g\cos\theta   \|^2
      -2  \Re\l \partial_\theta g \cos\theta,  g \sin\theta \r\\
    &\leq \sigma^{-1/2}\| g\sin\theta   \|^2 -\| g\cos\theta   \|^2+ \sigma^{1/2}\|\partial_\theta g\|^2,
  \end{align*}
  implying
  \begin{align}
    \sigma^{1/2}\| g\cos\theta \|^2\leq\| g\sin\theta   \|^2+\frac\sigma2 \|\partial_\theta g\|^2.
  \end{align}
  Combining the above estimate with \eqref{eq:spectral1}, we obtain
  the desired estimate \eqref{eq:spectralgap:proof}.
\end{proof}

We can now conclude the enhanced dissipation.
\begin{proof}[Proof of \cref{prop:hypocoercf}]
  Fix the constants \(a,b,c\) according to \cref{thm:hypo-constants}. Let $\lambda > 0$ to be fixed later independently of $\nu$. Let $\nu_0$ such that
  $\frac{4 a \nu_0}{c \lambda^2} \le 1$.
  From the definition of  $E[\psi]$, we have, for all $\nu \le \nu_0$,  at time $t = \lambda \nu^{-1/2}$
  \begin{align*}
    E[\psi] & \ge  \frac{1}{2} \| \psi \|^2 +  \frac{a \nu^{1/2} \lambda}{2} \|\nabla \psi\|^2
              + \frac{c \lambda^3 \nu^{-1/2}}{2} \| \nabla(p \cdot e) \psi \|^2 \\
            & \ge   \frac{1}{2} \| \psi \|^2  +  \frac{c \lambda^3 \nu^{-1/2}}{4} \left(  2 \| \nabla(p \cdot e) \psi \|^2 +   \frac{2 a \nu}{c \lambda^2} \|\nabla \psi\|^2 \right) \\
            & \ge  \frac12\left(  1 +   (ac)^{1/2} \lambda^2 \right)  \| \psi \|^2,
  \end{align*}
  where the last line comes from \cref{thm:hypo-interpolation}, with
  $\sigma := \frac{4 a \nu}{c \lambda^2} \le 1$.  By the evolution
  estimate of \cref{thm:hypo-constants} we also find (as \(\|\psi\|\)
  is non-increasing), that at time $t = \lambda \nu^{-1/2}$,
  \begin{equation*}
    E[\psi]  \le \frac 12 \| \psi^{in} \|^2
    +c \lambda^4\| \psi^{in} \|^2.
  \end{equation*}
  Hence we find that at time $t = \lambda \nu^{-1/2}$,
  \begin{equation*}
    \| \psi \|^2
    \le
    \frac{1 + 2c \lambda^4}{1  +   (ac)^{1/2} \lambda^2 }
    \| \psi^{in} \|^2.
  \end{equation*}
  Taking $\lambda$ small enough (depending on $a,b,c$ but
  independently of $\nu$), the factor at the right-hand side is less
  than \(1\), which implies exponential decay with a rate proportional
  to \(\nu^{1/2}\).
\end{proof}

\subsection{Hypocoercive estimate for vector fields}
\label{sub:enhancedVF}

In this subsection we introduce the vector fields to show the mixing
estimate \cref{thm:vector-field-decay} for \(\nu>0\) when the
semigroup \(\e^{t L_1}\) cannot solved explicitly. This has already
been used in a few instances for mixing estimates
\cite{CLN21,CZ20}. In terms of \(L_1\), the evolution of \(\psi\) of
\eqref{eq:ADEfourier}  can be written as
\begin{equation*}
  (\partial_t  - L_1) \psi = 0, \quad L_1  = - i e \cdot p  + \nu \Delta.
\end{equation*}
and the idea is to find a vector field \(J\) for which we control
\(J\psi\). A natural candidate for a vector field is
$J = \nabla + i t \nabla(p\cdot e)$, which commutes with the inviscid
part of the equation.  However, it does not commute well with the
diffusion operator. Namely, we find that
\begin{equation*}
  (\partial_t  - L_1) J\psi + \nu J \psi
  = 2 i \nu t \big(\nabla(p\cdot e) \psi + (p\cdot e) \nabla \psi \big).
\end{equation*}
By Duhamel's formula,
\begin{equation*}
  J\psi(t) = \e^{(L_1-\nu) t}  J\psi(0)
  + 2 i \nu  \int_0^t \e^{(L_1-\nu) (t-s)}
  \big( \nabla(p\cdot e) \psi(s)+ (p \cdot e) \nabla \psi(s) \big) \dd s.
\end{equation*}
If we were to rely only on the straightforward (yet optimal on time
scales $O(1)$) bounds
\begin{equation*}
  \|\psi(s)\|_{L^2} \le C, \quad  \| \nabla \psi(s) \|_{L^2} \le C (1+ s),
\end{equation*}
we would get
\begin{align*}
 \| J\psi(t)\|_{L^2} \le C \Big( 1 +  \nu \int_{0}^t   s (1+s)\, \dd s
  \Big),
  \qquad \forall t \le  \nu^{-1/2}.
\end{align*}
The second term behaves like $C \nu t^3$, and therefore diverges for
$t \gg \nu^{-1/3}$, which is a faster time scale than the enhanced
dissipation one for this problem.

To overcome this issue, we introduce the \emph{viscosity-adapted}
vector field \(J_\nu\) of the form
\begin{equation}
  J_\nu \psi = \alpha(t) \nabla \psi + i \beta(t) \nabla(p\cdot e) \psi.
\end{equation}
for scalar functions $\alpha = \alpha_\nu$, $\beta = \beta_\nu$. For
the evolution, we then find
\begin{align*}
  \Big(\partial_t + i (p\cdot e) - \nu \Delta \Big)
  &(\alpha \nabla \psi)
  = \alpha' \nabla \psi - i\alpha \nabla(p\cdot e) \psi - \nu \alpha \nabla \psi, \\
  \Big(\partial_t + i (p\cdot e) - \nu \Delta \Big)
  &(i \beta \nabla(p\cdot e) \psi)
    = i \beta' \nabla(p\cdot e) \psi
    - i \beta \nu   \nabla (p\cdot e)  \psi
    + 2 i \beta \nu \nabla \big( (p\cdot e) \psi \big).
\end{align*}
In order to control the commutator error terms, we set
\begin{equation*}
  \beta' = \alpha
  \quad \text{ and } \quad
  \alpha' = -2i\nu \beta
\end{equation*}
and thus take
\begin{equation*}
  \alpha(t) = \cosh(\sqrt{-2i\nu}\, t)
  \quad\text{ and }\quad
  \beta(t) = \frac{1}{\sqrt{-2i\nu}} \sinh(\sqrt{-2i\nu}\, t).
\end{equation*}
For the time frame \(t \le \nu^{-1/2}\) we see that \(\alpha \sim 1\)
and \(\beta \sim t\). By this choice, we find that now
\(X = J_\nu \psi\) solves
\begin{equation}\label{eq:evo-x}
  \Big(\partial_t + i (p\cdot e) - \nu \Delta \Big) X + \nu X   = 2i\beta\nu \nabla([p\cdot e-1] \psi)
\end{equation}
or
\begin{equation}\label{eq:evo-x-factor}
  \Big(\partial_t
  + i (p\cdot e) - \frac{2i\beta\nu}{\alpha} [p\cdot e -1]
  - \nu \Delta \Big)
  X + \nu X = -\frac{2\beta^2}{\alpha} \nu
  [p \cdot e - 1] \nabla(p\cdot e) \psi
  + 2 i \beta \nu \nabla(p\cdot e) \psi.
\end{equation}
The gain from the adapted vector field \(J_\nu\) is that the
right-hand side now vanishes at the north pole $p = e$. As the
enhanced dissipation estimate from \cref{sub:enhancedPS} provides
better decay properties for quantities that vanish at the poles
\(\pm e\), this will provide a better control of the source term, and
in turn a better control of $X$. Obviously, the right-hand side still
does not vanish at the south pole $p = -e$, so that we need to
localize the estimates away from this pole. Symmetrically, one could
construct another vector field $\tilde{X}$ for which the roles of the
north and south poles would be reversed.

To obtain good control of $X$ (away from the south pole), our starting
point is \cref{thm:hypo-constants} with $\delta=1$. After integration
in time, using that $\|\psi(t)\| \le \|\psi^{in}\|$ for all $t$, we
get for all $t_* \le \nu^{-1}$ that
\begin{equation} \label{inequality_integral_psi}
  \begin{aligned}
    & \sup_{t \in (0,t_*)}
      \Big( \|\psi\|^2 + \nu t \|\nabla \psi\|^2 + \nu t^3 \|\nabla (p \cdot e) \psi\|^2 \Big)\\
    & +  \int_0^{t_*}  \Big( \nu \|\nabla \psi\|^2
      + \nu t^2 \|\nabla(p \cdot e) \psi\|^2 + \nu^2 t \|\nabla^2 \psi\|^2
      + \nu^2 t^3 \|\nabla ( \nabla (p \cdot e) \psi)\|^2\Big)
      \lesssim (1+\nu^2 t_*^4) \|\psi^{in}\|^2.
  \end{aligned}
\end{equation}
From there, as a preliminary step, we deduce a few easy bounds on $X$
and $R$:
\begin{lemma}\label{thm:control-x-overall}
  There exists $\nu_0 > 0$, such that for all $0 < \nu \le \nu_0$ and
  all $t_* \le \nu^{-1}$ it holds that
  \begin{equation*}
    \int_0^{t_*} \nu \| X \|^2\, \dd t
    + \int_0^{t_*} \nu^2 t
    \| \nabla X \|^2\, \dd t
    \lesssim \sup_{[0,t_*]} \Big(|\alpha|^2 + \frac{|\beta|^2}{t^2}\Big)  (1+\nu^2 t_*^4)\| \psi^{in} \|^2.
  \end{equation*}
  Furthermore, for
  \begin{equation} \label{defi:R}
    R = 2i\beta\nu \nabla \big( [p\cdot e-1]  \psi \big)
  \end{equation}
  and any cutoff \(\chi\) excluding the south pole \(-e\) we have
  \begin{equation*}
    \int_0^{t_*} t \| R \chi \|^2\, \dd t
    \lesssim
    \sup_{[0,t^*]} \Big(|\alpha|^2 + \frac{|\beta|^2}{t^2}\Big)
    (1+\nu^2 t_*^4)\| \psi^{in} \|^2.
  \end{equation*}
  In particular,
   \begin{equation*}
    \int_0^{\nu^{-1/2}} \nu \| X \|^2\, \dd t
    + \int_0^{\nu^{-1/2}} \nu^2 t
    \| \nabla X \|^2\, \dd t +   \int_0^{\nu^{-1/2}} t \| R \chi \|^2\, \dd t
    \lesssim \| \psi^{in} \|^2.
   \end{equation*}
\end{lemma}
\begin{proof}
  The bounds relative to $X$ follow from the definition
  \(X = \alpha \nabla \psi + i \beta \nabla(p\cdot e) \psi\) and from
  the hypocoercive estimate \eqref{inequality_integral_psi}.  For the
  control of \(R\), we decompose
  \begin{equation*}
    t  \| R \chi \|^2
    \le 2 \nu^2 |\beta|^2 t  \left( \| [p\cdot e-1] \nabla \psi
      \chi \|^2 +  \| \psi \chi\|^2 \right)
    \le C \left(\frac{|\beta|^2}{t^2}\right)
    \nu^2 t^3   \left( \| \nabla(\nabla(p\cdot e) \psi) \chi \|^2 + \| \psi \chi\|^2 \right)
  \end{equation*}
  and the estimate also follows from \eqref{inequality_integral_psi}.
\end{proof}

We now state the key estimates of this paragraph, where we first focus
on times up to \(\nu^{-1/2}\). In the first iteration, we use the
nested cutoffs \(\chi\) and \(\chi'\) excluding the south pole
\(-e\).

\begin{lemma}\label{thm:x-hypo}
  There exists $\nu_0 > 0$, such that for all $0 < \nu \le \nu_0$ and
  cutoffs all $\chi, \chi':\bbS^2 \to [0,1]$ such that $\chi' = 1$ on
  the support of $\chi$ and such that $-e$ does not belong to the
  support of $\chi'$, one has for some $C > 0$ the estimate
  \begin{equation*}
    \begin{aligned}
      &\sup_{t _* \in [0,\nu^{-1/2}]}
        \Big(  \| X \chi \|^2 + \nu t \| \nabla X \chi \|^2  + \nu t^3 \| \nabla (p \cdot e) \otimes X \chi \|^2
        \Big) \\
      &+ \int_0^{ \nu^{-1/2}}
        \Big(
        \nu \| \nabla X \chi \| + \nu^2 \| \nabla^2 X \chi\|^2
        + \nu t^2 \| \nabla(p\cdot e)  \otimes  X \chi \|^2
        + \nu^2 t^3 \| \nabla(\nabla(p\cdot e) \otimes X) \chi \|^2
        \Big)\, \dd t \\
       &\qquad \le C \left( \|\psi^{in} \|_{H^1}^2 + \int_0^{\nu^{-1/2}}  \nu^2 t^3 \| \nabla(p\cdot e)  \otimes  X \chi' \|^2 \dd t\right).
    \end{aligned}
  \end{equation*}
\end{lemma}

Before the proof, we show that iterating this estimates shows the
estimate with the sharp rates.
\begin{corollary}\label{cor:localized}
  There exists $\nu_0 > 0$, such that for all $0 < \nu \le \nu_0$ and
  cutoff $\chi : \bbS^2 \to [0,1]$ such that $-e$ does not belong to
  the support of $\chi$, one has for some $C > 0$ the estimate
  \begin{equation*}
    \begin{aligned}
      &\sup_{t \in [0,\nu^{-1/2}]}
        \Big(  \| X \chi \|^2 + \nu t \| \nabla X \chi \|^2  + \nu t^3 \| \nabla (p \cdot e) \otimes X \chi \|^2
        \Big) \\
      &\qquad+ \int_0^{ \nu^{-1/2}}
        \Big(
        \nu \| \nabla X \chi \| + \nu^2 \| \nabla^2 X \chi \|^2
        + \nu t^2 \| \nabla(p\cdot e)  \otimes  X \chi \|^2
        + \nu^2 t^3 \| \nabla(\nabla(p\cdot e) \otimes X) \chi \|^2
      \Big)\, \dd t\\
      &\le C \|\psi^{in}\|_{H^1}^2.
    \end{aligned}
  \end{equation*}
\end{corollary}
\begin{proof}
  Consider nested cutoffs $\chi, \chi', \chi'' : \bbS^2 \to [0,1]$
  such that \(\chi'=1\) on \(\supp \chi\) and \(\chi''=1\) on
  \(\supp \chi'\) and \(\chi''\) excludes \(-e\).

  In a first step, applying \cref{thm:x-hypo} with $\chi'$ and
  \(\chi''\) yields in particular with the factor $\nu^{1/2}$ the
  bound
  \begin{align*}
    \int_0^{ \nu^{-1/2}} \nu^{3/2} t^2 \| \nabla(p\cdot e)  \otimes  X
    \chi' \|^2
    & \le C \nu^{1/2} \|\psi^{in}\|_{H^1}^2
      + C  \nu^{1/2} \int_0^{\nu^{-1/2}}  \nu^2 t^3 \| \nabla(p\cdot e)  \otimes  X \chi'' \|^2 \\
    & \le C \nu^{1/2} \|\psi^{in}\|_{H^1}^2
      + C \nu \int_0^{\nu^{1/2}}  \| X\|^2
      \le C' \|\psi^{in}\|_{H^1}^2
  \end{align*}
  where the last inequality comes from
  \cref{thm:control-x-overall}. Hence we gain the control
  \begin{align*}
    \int_0^{ \nu^{-1/2}} \nu^2 t^3 \| \nabla(p\cdot e)  \otimes  X \chi' \|^2
    \le   \int_0^{ \nu^{-1/2}} \nu^{3/2} t^2 \| \nabla(p\cdot e)
    \otimes  X \chi' \|^2
    \le C' \|\psi^{in}\|_{H^1}^2,
  \end{align*}
  so that the corollary now follows by another application of
  \cref{thm:x-hypo}.
\end{proof}
We now come to the proof of \cref{thm:x-hypo}.
\begin{proof}[Proof of \cref{thm:x-hypo}]
  The basic idea is to reproduce the hypocoercive estimates of
  \cref{sub:enhancedPS}, where we have additional contributions from
  the cut-off $\chi$ and the remainder $R$. We rewrite
  \eqref{eq:evo-x} as
  \begin{equation} \label{eq:XR}
    (\partial_t + i p\cdot e - \nu \Delta ) X + \nu X = R.
  \end{equation}
  We find
  \begin{equation*}
    \begin{aligned}
      \frac 12 \ddt \| X \chi\|^2 + \nu \|X \chi\|^2
      &= \nu \Re \l X \chi, \Delta X \chi\r
      + \Re \l X \chi, R \chi\r \\
      &\le - \nu \| \nabla X \chi\|^2
      + 2 \nu \| \nabla X \chi\|\, \| \nabla \chi \otimes X \|  + \Re \l X \chi, R \chi\r.
    \end{aligned}
  \end{equation*}
  For the last term, we use the explicit expression \eqref{defi:R} of
  $R$ and integrate by parts to find
  \begin{align*}
    \Re \l X \chi, R \chi\r
    & \le 2 |\beta| \nu \big( 2 \|\nabla \chi X\|  +  \| \chi \nabla X \| \big)
      \|(1- p \cdot e) \psi \chi\|
      \le C  \nu |\beta| \big(  \|\nabla \chi X\|  +  \| \chi \nabla X \| \big)
      \| \nabla (p \cdot e) \psi \|.
  \end{align*}
  It follows from Young's inequality that
  \begin{equation}
    \label{eq:x-l2}
    \begin{aligned}
      \frac 12 \ddt \| X \chi\|^2 +  \frac{\nu}{2} \| \nabla X \chi\|^2 \le  E_1(t),
    \end{aligned}
  \end{equation}
  where, for some absolute constant $C$,
  \begin{equation}
    E_1(t) := C \big( \nu \|X \|^2 + \nu |\beta|^2 \| \nabla (p \cdot e) \psi \|^2 \big).
  \end{equation}
  For the gradient, we find, using the commutation
  $\nabla \Delta X = \Delta \nabla X + O(|\nabla X|)$, see
  \cref{app:covariant}:
  \begin{equation*}
    \begin{aligned}
      \frac 12 \ddt \| \nabla X \chi\|^2 + \nu \|\nabla X \chi\|^2
      &= \nu \Re \l \nabla X \chi, \nabla \Delta X \chi\r
      - \Re \l  \nabla X \chi, i \nabla(p\cdot e) \otimes X \chi\r
      + \Re \l \nabla X \chi, \nabla R \chi\r \\
      &\le - \nu \| \nabla \nabla X \chi\|^2
      + 2 \nu    \| \nabla \chi  \otimes \nabla  X \| \| \nabla \nabla X \chi\|\, +  C \nu \| \nabla X \chi\|^2 \\
      &\quad +  \| \nabla X \chi\| \|\nabla(p\cdot e) \otimes   X \chi\|
      + \| \Delta X \chi\| \, \| R \chi\|
      +2 \| \nabla X \chi\| \, \| R \otimes  \nabla \chi\|.
    \end{aligned}
  \end{equation*}
  Hence, by Young's inequality,
  \begin{equation}
    \label{eq:x-grad}
    \begin{aligned}
      \frac 12 \ddt \| \nabla X \chi\|^2 +  \frac{\nu}{2} \| \nabla \nabla X \chi\|^2
      & \le    \| \nabla X \chi\| \|\nabla(p\cdot e)  \otimes X \chi\|   + E_2(t),
    \end{aligned}
  \end{equation}
  where, for some absolute constant $C$,
  \begin{equation}
    E_2(t) := C \left( \nu \|\nabla X\|^2 + \frac{1}{\nu} \|R \chi'\|^2\right).
  \end{equation}
  For the mixed term we find
  \begin{equation*}
    \begin{aligned}
      &\ddt \Re \l i \nabla(p\cdot e) \otimes X \chi, \nabla X \chi\r + 2 \nu \Re \l i \nabla(p\cdot e) \otimes X \chi, \nabla X \chi\r
      + \| \nabla(p\cdot e) \otimes X \chi \|^2 \\
      &= \nu \Re \l i \nabla(p\cdot e)  \otimes \Delta X \chi, \nabla X
      \chi\r
      + \nu \Re \l i \nabla(p\cdot e) \otimes X \chi, \nabla \Delta X
      \chi\r\\
      &\quad
        + \Re \l i \nabla(p\cdot e) \otimes  R \chi, \nabla X \chi\r
        + \Re \l i \nabla(p\cdot e) \otimes X \chi, \nabla R \chi\r
      \\
      &\le C_0 \nu \| \Delta X \chi\|
      \left(
        \| \nabla(p\cdot e) \otimes \nabla X \chi\|
        + \| X \chi\| + \| \nabla \chi \cdot \nabla(p\cdot e) X \|
      \right)\\
      &\quad
      + C_0 \left( \| \nabla (\nabla(p\cdot e) \otimes   X)  \chi\|
        + \| X \chi\| \right) \| R \chi\|
      + C_0 \| \nabla(p\cdot e) \otimes   X \chi\|\, \| \nabla \chi \otimes   R \|.
    \end{aligned}
  \end{equation*}
  It follows by Young's inequality that
  \begin{equation*}
    \begin{aligned}
      &\ddt \Re \l i \nabla(p\cdot e) \otimes X \chi, \nabla X \chi\r
      + \| \nabla(p\cdot e) \otimes X \chi \|^2 \\
      &\le C_0 \nu \| \Delta X \chi\|
      \left(
        \| \nabla \big(\nabla  (p\cdot e) \otimes X \big) \chi\|
        + \| X \chi\|   \right)  + E_3(t),
    \end{aligned}
  \end{equation*}
  where
  \begin{equation}
    \begin{aligned}
      E_3(t)
      & :=   C_0 \nu \| \Delta X \chi\|
      \| \nabla \chi \cdot \nabla(p\cdot e) X \|
      + C_0 \left( \| \nabla\cdot(\nabla(p\cdot e) \otimes   X)  \chi\|
        + \| X \chi\| \right) \| R \chi'\|.
    \end{aligned}
  \end{equation}
  Finally, we find using the explicit calculation of
  \cref{thm:laplace-mixed-term} that
  \begin{equation*}
    \begin{aligned}
      &\frac 12 \ddt \| \nabla(p\cdot e) \otimes X \chi \|^2
      + \nu \|\nabla(p\cdot e) \otimes X \chi\|^2  \\
      &=  \nu \Re \l \nabla(p\cdot e) \otimes X \chi ,
      \Delta(\nabla(p\cdot e) \otimes X) \chi \r
      + \nu \Re \l \nabla(p\cdot e) \otimes X \chi ,
      \nabla(p\cdot e) \otimes X \chi \r \\
      &\quad
      + 2 \nu \Re \l \nabla(p\cdot e) \otimes X \chi ,
      (p\cdot e) \nabla X  \chi \r
      + \Re \l \nabla(p\cdot e) \otimes X \chi, \nabla(p\cdot e) \otimes R \chi \r.
    \end{aligned}
  \end{equation*}
  For the contribution of \(R\), use the explicit expression of \(R\)
  and the definition of \(X\) to find that
  \begin{equation*}
    \begin{aligned}
      \nabla(p\cdot e) \otimes R
      &= \nabla(p\cdot e) \otimes 2 i \beta \nu \nabla( [p\cdot e -1] \psi )\\
      &= 2 i \nu \nabla( [p\cdot e - 1] (X - \alpha \nabla \psi))
        - 2 i \beta \nu \nabla^2(p\cdot e) [p\cdot e - 1]\psi,
    \end{aligned}
  \end{equation*}
  so that
  \begin{equation*}
    \begin{aligned}
      &\Re \l \nabla(p\cdot e) \otimes X \chi, \nabla(p\cdot e) \otimes R \chi \r\\
      &\le C \nu
      \Big( \| \nabla(\nabla(p\cdot e) \otimes X) \chi \|
      + \| \nabla(p\cdot e) \otimes X \chi \| \Big)
      \Big(
      \| \nabla(p\cdot e) \otimes X  \chi'\|
      + |\alpha| \| \nabla(\nabla(p\cdot e) \psi) \|
      + |\alpha| \| \psi \|
      \Big) \\
      &\quad+ C \nu |\beta| \| \nabla(p\cdot e)  \otimes X  \chi \|
      \| \nabla(p \cdot e) \psi \|.
    \end{aligned}
  \end{equation*}
  Hence we find
  \begin{align*}
    &\frac 12 \ddt \| \nabla(p\cdot e) \otimes X \chi \|^2
    + \nu \|\nabla(p\cdot e) \otimes X \chi\|^2\\
    &\qquad\qquad   \le  - \nu \| \nabla (\nabla(p\cdot e) \otimes X) \chi\|^2  + 2 \nu \|\nabla(p\cdot e) \otimes X \chi \|^2   + C_1 \nu \|X \chi\|^2  + E_4(t),
  \end{align*}
  where
  \begin{equation*}
    \begin{aligned}
      E_4(t)
      &:= C_1 \nu \| \nabla ( \nabla(p\cdot e) \otimes X )
        \chi\|
        \Big(
        \| \nabla(p\cdot e) \otimes X \chi' \|
        + |\alpha| \| \nabla(\nabla(p\cdot e) \psi) \|
        + |\alpha| \| \psi \|
        \Big)\\
      &\quad +  C_1 \nu \| \nabla(p\cdot e) \otimes X  \chi \|
        \Big(
        \| X \|
        +
        \| \nabla(p\cdot e) \otimes X \chi' \|
        + |\alpha| \| \nabla(\nabla(p\cdot e) \psi) \|
        + |\alpha| \| \psi \|
        \Big)\\
      &\quad + C_1 \nu |\beta| \| \nabla(p\cdot e) \otimes X \chi \|
        \| \nabla(p \cdot e) \psi \|.
    \end{aligned}
  \end{equation*}
  From here, we can proceed exactly as in the proof of
  \cref{thm:hypo-constants}: for suitable positive constants
  $a,b,c,\nu_0$, for all $\nu \le \nu_0$, an energy of the form
  \begin{align*}
    E[X] :=  E[\psi] =
    \frac12 \Big[\|X\chi\|^2+a \nu t \|\nabla X \chi\|^2
    +2b \nu t^2  \Re \l i \nabla(p\cdot e) \otimes X  \chi ,\nabla X  \chi \r
    + c \nu t^3 \| \nabla(p \cdot e) X \chi \|^2\Big]
  \end{align*}
  is both coercive
  \begin{equation*}
    E[X] \ge  \frac12 \Big[\|X\|^2+ \frac{a}{2} \nu t \|\nabla X
    \chi\|^2
    + \frac{c}{2} \nu t^3 \| \nabla(p \cdot e) X \chi \|^2\Big]
  \end{equation*}
  and satisfies for times \(t\le \nu^{-1}\) the inequality
  \begin{equation} \label{estim_E(X)}
    \begin{aligned}
      &    \ddt E[X]
      + \frac{\nu}{2} \| \nabla X \chi\|^2
      +  \frac{a\nu^2 t}{2} \| \nabla \nabla X \chi \|^2
      + \frac{b\nu t^2}{2}
      \| \nabla(p\cdot e) \otimes X \chi\|^2
      + \frac{c \nu^2 t^3}{2}
      \| \nabla(\nabla(p\cdot e) \otimes X) \chi\|^2 \\
      & \le C \nu^2 t^3 \| X \chi\|^2  + E_1(t) + a \nu t E_2(t) + b \nu t^2 E_3(t) + c \nu t^3 E_4(t).
    \end{aligned}
  \end{equation}
  Thanks to \eqref{inequality_integral_psi} and
  \cref{thm:control-x-overall}, we have for all $t_* \le \nu^{-1}$
  that
  \begin{equation} \label{E1E2}
    \int_0^{t_*}  E_1 + \int_0^{t_*} \nu t E_2
    \lesssim \sup_{[0,t_*]} \Big(|\alpha|^2 +
    \frac{|\beta|^2}{t^2}\Big)
    (1+\nu^2 t_*^4)\| \psi^{in} \|^2.
  \end{equation}
  In particular,
  \begin{equation*}
    \int_0^{\nu^{-1/2}}  E_1 +  \int_0^{\nu^{-1/2}} \nu t E_2 \lesssim 1.
  \end{equation*}
  Regarding $E_3$, we have by Young's inequality that for all $\kappa > 0$:
  \begin{equation*}
    \begin{aligned}
      \nu t^2  E_3
      & \le  \kappa \nu^2 t \| \nabla \nabla X \|^2
        + C_\kappa  \nu^2 t^3 \| \nabla(p\cdot e) \otimes   X \chi'\|^2  \\
      &\quad + \kappa \nu^2 t^3  \| \nabla (\nabla(p\cdot e) \otimes
        X)  \chi\|^2
        + \kappa \nu^2 t^3 \|X \chi\|^2 + C_\kappa  t \|R \chi'\|^2  \\
      & \le  \kappa \big(  \nu^2 t \| \nabla \nabla X \|^2
        + \nu^2 t^3  \| \nabla (\nabla(p\cdot e) \otimes   X)  \chi\|^2  \big)  \\
      &\quad + C_\kappa \nu^2 t^3   \| \nabla(p\cdot e) \otimes   X
        \chi'\|^2
        + \nu^2 t^3 \|X \chi\|^2  +  \tilde{E}_3,
    \end{aligned}
  \end{equation*}
  where still using \cref{thm:control-x-overall}, we have for all
  $t_* \le \nu^{-1}$,
  \begin{equation} \label{tildeE3}
    \int_0^{t}  \tilde{E}_3 \lesssim   \int_0^{t_*} t \|R\chi'\|^2 \dd t
    \lesssim
    \sup_{[0,t_*]} \Big(|\alpha|^2 + \frac{|\beta|^2}{t^2}\Big)  (1+\nu^2 t_*^4)\| \psi^{in} \|^2.
  \end{equation}
  In particular,
  \begin{equation*}
    \int_0^{\nu^{-1/2}}  \tilde{E}_3 \lesssim   \int_0^{\nu^{-1/2}} t
    \|R \chi'\|^2 \lesssim 1.
  \end{equation*}
  Regarding $E_4$, we estimate for all $t \le \nu^{-1}$
  \begin{equation*}
    \begin{aligned}
      \nu t^3  E_4  & \le \kappa \nu^2 t^3 \| \nabla ( \nabla(p\cdot
      e) \otimes X ) \chi\|^2\\
      &\quad + C_\kappa \nu^2 t^3 \Big(
      \| \nabla(p\cdot e) \otimes   X \chi'\|^2
      + |\alpha|^2 \| \nabla(\nabla(p\cdot e) \psi) \|^2
      + |\alpha|^2  \| \psi \|^2
      \Big)
      \\
      & \quad + \kappa \nu t^2  \| \nabla(p\cdot e) \otimes X  \chi
      \|^2 + C_\kappa \nu^3 t^4 \|X\|^2
      + C_\kappa \nu^3 t^4 |\beta|^2 \| \nabla(p\cdot e) \psi \|^2
      \\
      & \le \kappa \big( \nu^2 t^3 \| \nabla ( \nabla(p\cdot e) \otimes X ) \chi\|^2  + \nu t^2  \| \nabla(p\cdot e) \otimes X  \chi \|^2 \big) + C_\kappa \nu^2 t^3 \| \nabla(p\cdot e) \otimes   X \chi'\|^2 +  \tilde{E}_4,
    \end{aligned}
  \end{equation*}
  where, for all $t_* \le \nu^{-1}$
  \begin{equation} \label{tildeE4}
    \begin{aligned}
      \int_0^{t_*}  \tilde{E}_4
      &  \lesssim  \nu^2 t_*^4 \int_0^{t_*} \nu \|X\|^2 \\
      & \quad +   \int_0^{t_*} \Big(  \nu^2 t^3 |\alpha|^2 \|
        \nabla(\nabla(p\cdot e) \psi) \chi' \|^2  + \nu^2 t^3 |\alpha|^2 \|
        \psi \chi' \|^2 + \nu^3 t^4 |\beta|^2 \| \nabla(p\cdot e) \psi \|^2
        \Big)\\
      &\lesssim \sup_{[0,t_*]} \Big(|\alpha|^2 + \frac{|\beta|^2}{t^2}\Big)  (1+\nu^4 t_*^8)\| \psi^{in} \|^2.
    \end{aligned}
  \end{equation}
  Here we used in the last step \eqref{inequality_integral_psi} and
  \cref{thm:control-x-overall}. In particular, this shows
  \begin{equation*}
    \int_0^{\nu^{-1/2}}  \tilde{E}_4    \lesssim 1.
  \end{equation*}
  For $\kappa$ small enough independently of $\nu$, and for all
  $t \le \nu^{-1}$, we get
  \begin{equation} \label{estim_prelim_E(X)}
    \begin{aligned}
      &    \ddt E[X]
      + \frac{\nu}{4} \| \nabla X \chi\|^2
      +  \frac{a\nu^2 t}{4} \| \nabla \nabla X \chi \|^2
      + \frac{b\nu t^2}{4}
      \| \nabla(p\cdot e) \otimes X \chi\|^2
      + \frac{c \nu^2 t^3}{4}
      \| \nabla(\nabla(p\cdot e) \otimes X) \chi\|^2 \\
      &\le C' \nu^2 t^3 \| X \chi\|^2
        + C'  \nu^2 t^3 \| \nabla(p\cdot e) \otimes   X \chi'\|^2
        + E_1(t) + a \nu t E_2(t) + b \tilde{E}_3(t) + c \tilde{E}_4(t).
    \end{aligned}
  \end{equation}
  The desired inequality follows by time integration from $0$ to
  $\nu^{-1/2}$, using the bounds on
  $E_1, E_2, \tilde{E}_3, \tilde{E}_4$.
\end{proof}

To obtain a good decay rate for the velocity field of our linearized
model, we will need to control also the \(L^\infty\) norm of $\psi$
near the pole. This is the purpose of the next lemma.
\begin{lemma}\label{thm:x-infty}
  There exist absolute constants $C > 0$ and $\nu_0 > 0$ such that for
  all $\nu \le \nu_0$ and \(t \le \nu^{-1/2}\) it holds that
  \begin{equation*}
    |X(t,p)|^2     \le  \|\nabla \psi^{in}\|^2_{L^\infty} + C
    \left([1-p\cdot e]^3 \nu t^4  + [1-p\cdot e]^2 \nu^2 t^5
      +  [1-p\cdot e] \nu t^2 +  1 \right) \|\psi^{in}\|_{L^\infty}^2.
  \end{equation*}
\end{lemma}
\begin{proof}
  Using the Leibniz rule for the covariant derivative, we have
  \begin{equation*}
    \frac{1}{2} \Delta |X|^2 =   g (\Delta X, X ) + |\nabla X|^2,
  \end{equation*}
  where $\Delta X$ still refers to the connection Laplacian of the
  vector field $X$ and where $|X| = \sqrt{g(X,X)}$ is the usual norm
  induced by the metric on tensors.  Back to \eqref{eq:evo-x-factor},
  it follows that
  \begin{equation} \label{eq_X2}
    (\partial_t - \nu \Delta) \frac{|X|^2}{2} + \nu |\nabla X|^2
    + \nu |X|^2
    = g (-\frac{2\beta^2}{\alpha} \nu
    [p \cdot e - 1] \nabla(p\cdot e) \psi   + 2 i \beta \nu
    \nabla(p\cdot e) \psi, X )
  \end{equation}
  so that
  \begin{equation*}
    \begin{aligned}
      (\partial_t - \nu \Delta) \frac{|X|^2}{2}
      & \le  \nu t |X|^2
        + C \nu \frac{|\beta|^4}{|\alpha|^2 t^4} t^3
        [1- p \cdot e]^3 |\psi|^2
        +  C \frac{|\beta|^2}{t^2} \nu t [1- p \cdot e] |\psi|^2 \\
      & \le  \nu t |X|^2 + C'  \Big(\nu  t^3 [1-p\cdot e]^3
        +   \nu t [1-p\cdot e]\Big) \|\psi^{in}\|_{L^\infty}^2,
    \end{aligned}
  \end{equation*}
  where the latter inequality holds for all $t \le \nu^{-1/2}$. It is
  then easy to find absolute constants $c_0 > 0$, $\nu_0 > 0$,
  $a_0, a_1, a_2, a_3 > 0$, such that for all $\nu \le \nu_0$, the
  function
  \begin{equation*}
    f = \Big( a_0 \nu  t^4 [1-p\cdot e]^3
    + a_1 \nu^2 t^5 [1-p\cdot e]^2
    + a_2 \nu t^2 [1-p\cdot e] + a_3 \Big) \|\psi^{in}\|_{L^\infty}^2
  \end{equation*}
  satisfies
  $$
  (\partial_t - \nu \Delta) \Big( \e^{-c_0 \nu t^2} |X|^2 -  f \Big)
  \le  0   \qquad  \forall  t \le \nu^{-1/2}.
  $$
  The lemma then follows from the maximum principle for the scalar
  heat flow on the sphere.
\end{proof}

The estimates on \(X=J_\nu \psi\) can roughly yield a decay like
\(t^{-1}\). For the sharp decay of \(t^{-2}\) for the velocity field,
we need to iterate and thus also need a control of
$J_\nu X = \alpha \nabla X + i \beta \nabla (p \cdot e) \otimes
X$. From \eqref{eq:evo-x}, we get
\begin{equation*}
  \Big(\partial_t + i (p\cdot e) - \nu \Delta \Big) J_\nu X + \nu
  J_\nu X
  = - \nu  \alpha [\Delta, \nabla] X
  - i \beta \nu \nabla (p \cdot e) \otimes X
  + 2 i \beta \nu \nabla \big( [p \cdot e - 1] X )
  + J_\nu R,
\end{equation*}
where we recall \(R=2i\beta\nu \nabla([p\cdot e-1] \psi)\). By Ricci's
formula, see \cref{app:covariant}, one has
$[\Delta, \nabla] X = \mathcal{R} \nabla X$ for some tensor
$\mathcal{R}$ so that
\begin{equation*}
  - \nu  \alpha [\Delta, \nabla] X
  = - \nu  \mathcal{R} J_\nu X
  + \nu i \beta \mathcal{R}  \nabla (p \cdot e) \otimes X
\end{equation*}
and eventually
\begin{equation}\label{eq:evo-jx}
  \Big(\partial_t + i (p\cdot e) - \nu \Delta \Big) J_\nu X
  =  - \nu (\mathcal{R} + 1) J_\nu X
  + \nu i \beta (\mathcal{R} - 1)  \nabla (p \cdot e) \otimes X
  + 2 i \beta \nu \nabla \big( (p \cdot e - 1) X )   + J_\nu R.
\end{equation}
Using the previous estimates, we shall prove:
\begin{lemma}\label{thm:jx-control}
  There exists $\nu_0 > 0$, such that for all $0 < \nu \le \nu_0$ and
  cutoffs $\chi : \bbS^2 \to [0,1]$ such that $-e$ does not belong to
  the support of $\chi$, one can find $C > 0$ independent of $\nu$
  satisfying for all \(t \le \nu^{-1/2}\)
  \begin{equation*}
    \| J_\nu X(t)  \chi \|^2 \le C \|\psi^{in}\|_{H^2}^2.
  \end{equation*}
\end{lemma}
\begin{proof}
  Let $\chi, \chi'$ such that $\chi' = 1$ on the support of $\chi$ and
  such that $-e$ does not belong to the support of $\chi'$. Performing
  an $L^2$ estimate on \eqref{eq:evo-jx}, we find after standard
  integrations by parts that
  \begin{equation*}
    \begin{aligned}
      \frac{1}{2}  \ddt \| J_\nu X \chi \|^2
      +  \nu \| \nabla J_\nu X \chi\|^2
      & \le 2 \nu   \|\nabla J_\nu X \chi\| \,
        \|J_\nu X \otimes \nabla \chi \|
        + C \nu  \| J_\nu X  \chi\|^2 \\
      &\quad + C \nu |\beta| \| \nabla(p\cdot e) \otimes  X  \chi \|  \| J_\nu X  \chi\|    \\
      &\quad + 2 \nu |\beta| \| (p\cdot e - 1) X  \chi \|
        \Big( \|\nabla J_\nu X  \chi\| +  2 \| J_\nu X  \otimes \nabla \chi\| \Big)\\
      &\quad + \Re \l J_\nu R \chi, J_\nu X \chi \r.
    \end{aligned}
  \end{equation*}
  It follows  that
  \begin{equation*}
    \frac{1}{2} \ddt \| J_\nu X \chi \|^2
    + \frac{\nu}{2} \| \nabla J_\nu X \chi\|^2
    \le C' \nu \|J_\nu X \chi'\|^2
    + C' \nu |\beta|^2  \| \nabla (p \cdot e) \otimes X  \chi \|^2
    + \l J_\nu R \chi, J_\nu X \chi \r.
  \end{equation*}
  We first notice that
  \begin{equation*}
    \begin{aligned}
      \nu \| JX \chi'\|^2  +   \nu \beta^2  \| \nabla (p \cdot e)
      \otimes X  \chi \|^2
      & \lesssim
        \nu |\alpha|^2 \| \nabla X \chi'\|^2
        + \nu |\beta|^2  \| \nabla(p\cdot e) \otimes X  \chi' \|^2.
    \end{aligned}
  \end{equation*}
  For the last term, we use that
  \begin{align*}
    J_\nu R
    & = 2 i \nu \beta J_\nu \nabla (  [p \cdot e - 1] \psi)
      = 2 i \nu \beta  \nabla (  [p \cdot e - 1] J_\nu \psi)
      + 2 i \nu \beta [J_\nu ,  \nabla (  [p \cdot e - 1] \cdot ) ] \psi \\
    & = 2 i \nu \beta  \nabla (  [p \cdot e - 1] X)
      + 2 i \nu \beta [J_\nu ,  \nabla (  [p \cdot e - 1] \cdot ) ] \psi
      =: R_1 + R_2.
  \end{align*}
  For the first part, we find by integration by parts that
  \begin{align*}
    \Re \l R_1 \chi  , J_\nu X \chi \r
    \le 2 \nu |\beta|\, \| (p\cdot e - 1) X  \chi \|
    \Big( \|\nabla J_\nu X  \chi\| +  2 \| J_\nu X  \otimes \nabla \chi\| \Big).
  \end{align*}
  This right-hand side was already encountered above and we bound it
  as
  \begin{align*}
    \Re \l R_1 \chi  , J_\nu X \chi \r
    \le \frac{\nu}{8} \| \nabla J_\nu X \chi\|^2
    + \nu \|   J_\nu X \chi' \|^2
    + C \nu |\beta|^2 \| \nabla (p \cdot e) \otimes X  \chi \|^2.
  \end{align*}
  For the other component, we further compute that
  \begin{equation*}
    R_2 =  2 i \nu \beta \alpha \nabla  \big( \nabla (p \cdot e) \psi
    \big)
    + 2 \nu \beta^2 \nabla (p \cdot e) (p \cdot e - 1) \psi
    := R_{2,1} + R_{2,2}.
  \end{equation*}
  One has
  \begin{align*}
    \Re \l R_{2,1} \chi  , J_\nu X \chi \r
    \le C \nu |\alpha \beta| \| \nabla (p \cdot e) \psi \chi \| \Big( \|\nabla
    J_\nu X  \chi\|
    +  \| J_\nu X  \otimes \nabla \chi\| \Big),
  \end{align*}
  so that one ends up with
  \begin{align*}
    \Re \l R_{2,1} \chi  , J_\nu X \chi \r
    \le \frac{\nu}{8} \| \nabla J_\nu X \chi\|^2
    + C \nu \|   J_\nu X \chi' \|^2
    + C \nu |\alpha \beta|^2  \| \nabla (p \cdot e) \psi \chi \|^2.
  \end{align*}
  Eventually, we notice that
  \begin{equation*}
    R_{2,2} = \frac{2}{i} \nu \beta (p \cdot e - 1)  (X  - \alpha \nabla \psi),
  \end{equation*}
  resulting after integrations by part in
  \begin{equation}
    \begin{aligned}
      \Re \l R_{2,2} \chi  , J_\nu X \chi \r
      & \le C \nu |\beta| \| (p \cdot e - 1) X \chi \| \|J_\nu
        X \chi\|
        + C \nu |\alpha \beta|  \| \nabla (p \cdot e) \psi \|
        \big(\| \nabla J_\nu X \chi \| +  \|J_\nu X \chi'\|\big) \\
      & \le \frac{\nu}{8} \| \nabla J_\nu X \chi\|^2
        + C \nu \| J_\nu X \chi' \|^2
        + C \nu |\beta|^2 \big(  \| \nabla (p \cdot e)  \otimes X \chi \|^2
        + |\alpha|^2 \|\nabla (p \cdot e) \psi \|^2 \big).
    \end{aligned}
  \end{equation}
  Hence, collecting all these bounds, we find for all $t \le \nu^{-1}$
  that
  \begin{equation}  \label{finalboundJ}
    \frac{1}{2} \ddt \| J_\nu X \chi \|^2
    + \frac{\nu}{8} \| \nabla J_\nu X \chi\|^2
    \lesssim E(t)
  \end{equation}
  where
  \begin{equation} \label{defiE(t)}
    E(t) :=
    \nu |\alpha|^2 \| \nabla X \chi'\|^2
    + \nu |\beta|^2  \| \nabla(p\cdot e) \otimes X  \chi' \|^2
    + \nu |\alpha \beta|^2 \|\nabla (p \cdot e) \psi \|^2.
  \end{equation}
  For $t \le \nu^{-1/2}$,
  \begin{equation*}
    E(t) \lesssim  \nu \| \nabla X \chi'\|^2
    + \nu t^2  \| \nabla(p\cdot e) \otimes X  \chi' \|^2
    + \nu t^2 \|\nabla (p \cdot e) \psi \|^2
  \end{equation*}
  and thanks to \cref{thm:x-hypo}, one has
  $\int_0^{\nu^{-1/2}} E(t)\dd t \le C$, independent of $\nu$. This
  yields the claimed estimate for \(J_\nu X\).
\end{proof}

\subsection{Mixing estimates} \label{sub:mixing}

We now show how the obtained estimates can be used to obtain the
mixing estimates up to time \(\nu^{-1/2}\) of
\cref{thm:vector-field-decay}.

\begin{proof}[Proof of \cref{thm:vector-field-decay} for times up to \(\nu^{-1/2}\)]
  We start with the proof of the first inequality. We split the
  integrand
  \begin{equation*}
    \int_{\bbS^2} \psi(t) F\, \dd p
    =  \int_{\bbS^2} \psi(t) F  \chi \, \dd p+  \int_{\bbS^2} \psi(t) F  (1-\chi)\, \dd p,
  \end{equation*}
  where $\chi$ is a smooth function which is $1$ near $e$ and $0$ near
  $-e$. By symmetry consideration, it is enough to show the decay for
  $\int_{\bbS^2} \psi(t) F \chi$. We introduce a cutoff
  \(\chi_{\epsilon}\) which is $1$ in a ball of radius \(\epsilon\)
  around the pole \(p=e\), zero outside a ball of radius $2\eps$, and
  satisfies $|\nabla \chi_\epsilon| \lesssim \epsilon^{-1}$. We write
  \begin{equation*}
    \begin{aligned}
      \int_{\bbS^2}  \psi F\chi \, \dd p
      & = \int_{\bbS^2}  \psi F \chi_\epsilon \chi \, \dd p+ \int_{\bbS^2}  \psi F (1- \chi_\epsilon) \chi \, \dd p=: I_1 + I_2.
    \end{aligned}
  \end{equation*}
  As $\|\psi\|_{L^\infty} \le \|\psi^{in}\|_{L^\infty}$, we get
  \begin{equation*}
    |I_1| \le C \epsilon^2  \|F\|_{L^\infty} \|\psi^{in}\|_{L^\infty}.
  \end{equation*}
  As regards $I_2$, we write
  \begin{equation}
    \begin{aligned}
      \int \psi F (1- \chi_\epsilon) \chi
      & =  \int \nabla(p\cdot e) \cdot \nabla(p\cdot e) \psi
        \frac{F}{|\nabla(p\cdot e)|^2} (1- \chi_\epsilon)  \chi \\
      & =  \frac{1}{i \beta}  \int \nabla(p\cdot e) \cdot (X-\alpha \nabla \psi)
        \frac{F}{|\nabla(p\cdot e)|^2} (1- \chi_\epsilon) \chi \\
      & =  \frac{1}{i \beta}  \int \nabla(p\cdot e) \cdot X   \frac{F}{|\nabla(p\cdot e)|^2} (1- \chi_\epsilon)  \chi \\
      &\quad +  \frac{\alpha}{i \beta}  \int  \psi \nabla \cdot  \left( \nabla(p\cdot e)    \frac{F}{|\nabla(p\cdot e)|^2} (1- \chi_\epsilon) \chi  \right) = I_{2,1} + I_{2,2}.
    \end{aligned}
  \end{equation}
 We introduce again spherical coordinates
$(\theta,\varphi)$, with colatitude $\theta \in (0, \pi)$ and longitude
$\varphi \in (0,2\pi)$, so that
$p = \sin \theta \cos \varphi\, e_x + \sin \theta \sin \varphi\, e_y +
\cos\theta\, e$, and $\nabla (p \cdot e) = - \sin \theta e_\theta$, while the surface measure on the sphere is $\dd s = \sin\theta \dd\theta \dd\varphi$.
  For the first term, we find  that
  \begin{equation*}
  \begin{aligned}
    |I_{2,1}|
  &  \le \frac{C}{|\beta|} \|F\|_{L^\infty}
    \int_{|\theta| \ge \epsilon} \frac{|X|}{|\sin \theta|} \chi \, \dd s
    \le \frac{C}{|\beta|} \|F\|_{L^\infty}
    \left( \int_{|\theta| \ge \epsilon} \frac{1}{|\sin \theta|^2}\,
      \dd s \right)^{\frac 12}  \|X \chi\| \\
 &  \le \frac{C'}{|\beta|} \|F\|_{L^\infty} \left(  \int_{|\theta| \ge \epsilon} \frac{1}{|\sin \theta|}\, \dd \theta \right)^{\frac 12}  \|X \chi\|     \le \frac{C''}{|\beta|} \|F\|_{L^\infty}
    \sqrt{|\ln \epsilon|}     \|X \chi\|   .
    \end{aligned}
  \end{equation*}
  In particular with \cref{cor:localized} this shows for times
  \(t \le \nu^{-1/2}\) that
  \begin{equation*}
    |I_{2,1}|
    \le \frac{C}{t} \|F\|_{L^\infty}  \sqrt{|\ln\epsilon|}\, \|\psi^{in}\|_{H^1}.
  \end{equation*}
  For the second term, we get
  \begin{align*}
    |I_{2,2}|
    &\le C \big|\frac{\alpha}{\beta}\big| \int_{|\theta| \ge \epsilon} |\psi| \left( |\nabla
      \chi_\epsilon| \frac{|F|}{|\sin\theta|} +  \frac{1}{|\sin\theta|^2} |F| + \frac{|\nabla F|}{|\sin\theta|} \right)
    \dd s \\
    &\le C' \big|\frac{\alpha}{\beta}\big|
      \|\psi^{in}\|_{L^\infty}  \bigg(\|F\|_{L^\infty}  \int_{2 \eps \ge  |\theta| \ge \epsilon} \frac{1}{\eps |\sin\theta|}  \dd s  + \|F\|_{L^\infty}  \int_{|\theta| \ge \epsilon}    \frac{1}{|\sin \theta|^2}  \dd s \\
      &  \hspace{3cm} + \|F\|_{H^1} \Big(  \int_{|\theta| \ge \epsilon} \frac{1}{|\sin \theta|^2} \dd s \Big)^{1/2} \bigg) \\
      & \le  C'' \big|\frac{\alpha}{\beta}\big|
      \|\psi^{in}\|_{L^\infty}  |\ln\epsilon|  \big(\|F\|_{L^\infty}  +  \|F\|_{H^1}\big) .
  \end{align*}
  In particular  for times
  \(t \le \nu^{-1/2}\) it shows that
  \begin{equation*}
    |I_{2,2}| \le \frac{C''}{t}
    |\ln\epsilon|\, \|\psi^{in}\|_{L^\infty}
    ( \|F\|_{L^\infty} + \|F\|_{H^1}).
  \end{equation*}
  Taking $\epsilon = \frac{1}{\sqrt{t}}$ yields the first inequality.

  As regards the second inequality, we write
  \begin{equation*}
    \begin{aligned}
      \int \psi F \nabla (p \cdot e)\chi
      &= \frac{1}{i\beta}
        \int  F (X-\alpha \nabla \psi) \chi
        = \frac{1}{i \beta}    \int  F X \chi + \frac{\alpha}{i \beta} \int \psi \nabla \cdot (F \chi) = J_1 + J_2.
    \end{aligned}
  \end{equation*}
  Applying the first inequality to the integral in $J_2$, we find for
  times \(t \le \nu^{-1/2}\) that
  \begin{equation*}
    |J_2| \le C  \frac{\sqrt{\ln(2+t)}}{t^2} \|\psi^{in}\|_{H^{1+\delta}} \|F\|_{H^{2+\delta}}.
  \end{equation*}
  The term $J_1$ is similar, except that $\psi$ is replaced by $X$, on
  which we have a weaker control. Using the same strategy as in the
  proof of the first inequality, we can write
  \begin{equation*}
    J_1 =  \frac{1}{i \beta} \int X F\chi
    = \frac{1}{i \beta}  \int X F \chi_\epsilon \chi
    + \frac{1}{i \beta}  \int X F (1- \chi_\epsilon) \chi
    =:  K_{1} + K_2.
  \end{equation*}
  We treat $K_1$ as $I_1$, except that we use \cref{thm:x-infty} as a
  substitute to the $L^\infty$ bound on $\psi$. We get
  \begin{equation*}
    |K_1| \le \frac{C}{t}
    \big( \epsilon^2 + \sqrt{\nu} t^2 \epsilon^5 + \nu t^{5/2} \epsilon^4 +
    \nu^{1/2} t \epsilon^3 \big)
    \|\psi^{in}\|_{W^{1,\infty}} \|F\|_{L^\infty}.
  \end{equation*}
  We then treat $K_2$ as $I_2$, resulting in
  \begin{equation}
    \begin{aligned}
      K_2  & =  \frac{1}{-\beta^2}  \int \nabla(p\cdot e) \cdot J_\nu X   \frac{F}{|\nabla(p\cdot e)|^2} (1- \chi_\epsilon)  \chi \\
           &\quad +  \frac{\alpha}{-\beta^2}  \int X \nabla \cdot  \left( \nabla(p\cdot e)    \frac{F}{|\nabla(p\cdot e)|^2} (1- \chi_\epsilon) \chi  \right) = K_{2,1} + K_{2,2}.
    \end{aligned}
  \end{equation}
  As before we find
  \begin{equation*}
    |K_{2,1}| \le
    \frac{C}{|\beta|^2} \|J_\nu X\| \sqrt{|\ln\epsilon|}
    \|F\|_{L^\infty},
  \end{equation*}
  which yields in particular for \(t \le \nu^{-1/2}\) that
  \begin{equation*}
    |K_{2,1}|
    \le \frac{C'}{t^2} \|\psi^{in}\|_{H^2} \|F\|_{L^\infty}
    \sqrt{|\ln\epsilon|}.
  \end{equation*}
  For  $K_{2,2}$, we further decompose it as
  \begin{align*}
    K_{2,2}
    & =  \frac{\alpha}{\beta^2}  \int X \nabla \chi_\epsilon \cdot \left( \nabla(p\cdot e)    \frac{F}{|\nabla(p\cdot e)|^2} \chi  \right) -  \frac{\alpha}{\beta^2}  \int X  (1-\chi_\epsilon)  \nabla \cdot \left( \nabla(p\cdot e)    \frac{F}{|\nabla(p\cdot e)|^2} \chi  \right) \\
    & =  \frac{\alpha}{\beta^2}  \int  X \nabla \chi_\epsilon \cdot \left( \nabla(p\cdot e)    \frac{F}{|\nabla(p\cdot e)|^2} \chi  \right)  \\
    &\quad -  \frac{\alpha}{i \beta^3}     \int J_\nu X (1-\chi_\epsilon)  \frac{\nabla (p \cdot e)}{|\nabla (p \cdot e)|^2} \nabla \cdot \left( \nabla(p\cdot e)
      \frac{F}{|\nabla(p\cdot e)|^2} \chi  \right)   \\
    &\quad - \frac{\alpha^2}{i \beta^3}   \int X \nabla \cdot \left( (1-\chi_\epsilon)  \frac{\nabla (p \cdot e)}{|\nabla (p \cdot e)|^2} \nabla \cdot \left( \nabla(p\cdot e)
      \frac{F}{|\nabla(p\cdot e)|^2} \chi  \right)  \right) =: H_1 + H_2 + H_3.
  \end{align*}
  For $H_1$ we obtain the bound
  \begin{equation*}
    |H_1| \le \frac{C}{\epsilon}\, \frac{|\alpha|}{|\beta|^2}\, \|F\|_{L^\infty}
    \int_{\epsilon \le |\theta| \le c \epsilon} |X|
  \end{equation*}
  and in particular for times \(t \le \nu^{-1/2}\) we get by
  \cref{thm:x-infty} that
  \begin{equation*}
    |H_1| \le \frac{C}{t^2 \epsilon}
    \big( \epsilon + \sqrt{\nu} t^2 \epsilon^4 + \nu t^{5/2}
    \epsilon^3 + \nu^{1/2} t \epsilon^2 \big)
    \|F\|_{L^\infty} \|\psi^{in}\|_{W^{1,\infty}}.
  \end{equation*}
  For $H_2$, we compute
  \begin{equation*}
    |H_2| \le C\, \frac{|\alpha|}{|\beta|^3} \|F\|_{W^{1,\infty}}
    \int_{|\theta| \ge \epsilon} \frac{|J_\nu X|}{|\sin(\theta)|^{2}}
    \,\dd\theta\, \dd \varphi
    \le \frac{C}{\epsilon^2}\, \frac{|\alpha|}{|\beta|^3} \|F\|_{W^{1,\infty}}
    \| J_\nu X \|.
  \end{equation*}
  In particular for \(t \le \nu^{-1/2}\) it shows by
  \cref{thm:jx-control} that
  \begin{equation*}
    |H_2| \le \frac{C'}{t^3 \epsilon^2}  \|F\|_{W^{1,\infty}} \|\psi^{in}\|_{H^2}.
  \end{equation*}
  Finally,
  \begin{align*}
    |H_3|
    \le C \frac{|\alpha|^2}{|\beta|^3} \int_{|\theta| \ge \epsilon} |X|
    \left(
      |\nabla \chi_\epsilon|
      \frac{(|F| + |\nabla F|)}{|\sin \theta|^2}
      + \frac{(|F| + |\nabla F|)}{|\sin \theta|^3}
      + \frac{|\nabla^2 F|}{|\sin \theta|}
    \right).
  \end{align*}
  In particular, for \(t \le \nu^{-1/2}\) we use \cref{thm:x-infty} to
  find
  \begin{equation*}
    |H_3|
    \le \frac{C}{t^3}
    \Big( \epsilon^{-2} + \sqrt{\nu} t^2 + \nu t^{5/2} (1+|\ln \epsilon|)  +
    \nu^{1/2} t \epsilon^{-1} \Big)
    \|\psi^{in}\|_{W^{1,\infty}}  \|F\|_{H^{2+\delta}}.
  \end{equation*}
  Taking $\epsilon = \frac{1}{\sqrt{t}}$ yields the second inequality.
\end{proof}

\subsection{Mixing estimates after time \(\nu^{-1/2}\)} \label{sub:mixing-long}

In this subsection, we prove the remaining estimates in
\cref{thm:vector-field-decay}, i.e.\ for times
\(t \in [\nu^{-1/2},c\nu^{-1/2}|\ln \nu|]\). The general idea is the
same and it has already been anticipated in the previous
subsections. The main difference is to notice that we cannot
approximate anymore $\alpha$ by $1$ and $\beta$ by $t$ but
that \(\alpha\) and \(\beta\) are exponentially growing in this time
scale.

\begin{proof}[Proof of \cref{thm:vector-field-decay} for times after
  \(\nu^{-1/2}\)]

  Performing the proof of \cref{thm:x-hypo}, we do not substitute
  \(\alpha\) and \(\beta\) in the last step but rather take out a
  supremum. As an analogue of \cref{cor:localized} this yields
  for all $t_* \le \nu^{-1}$ the bound
  \begin{equation} \label{X_long}
    \begin{aligned}
      &\sup_{t \in [0,t_*]}
        \left(  \| X \chi \|^2 + \nu t \| \nabla X \chi \|^2  + \nu t^3 \| \nabla (p \cdot e) \otimes X \chi \|^2
        \right) \\
      &+ \int_0^{t_*}
        \left(
        \nu \| \nabla X \chi \| + \nu^2 \| \nabla^2 X \chi \|^2
        + \nu t^2 \| \nabla(p\cdot e)  \otimes  X \chi \|^2
        + \nu^2 t^3 \| \nabla(\nabla(p\cdot e) \otimes X) \chi \|^2
        \right)\, \dd t  \\
      &  \le C \sup_{[0,t^*]} \Big(|\alpha|^2 +   \frac{|\beta|^2}{t^2}\Big)  (1+\nu^4 t_*^8)\| \psi^{in} \|_{H^1}^2.
    \end{aligned}
  \end{equation}
  Using this inequality in \eqref{finalboundJ}-\eqref{defiE(t)}, we
  find as an analogue of \cref{thm:jx-control} that for all
  $t_* \le \nu^{-1}$ the bound
  \begin{equation}  \label{JX_long}
    \sup_{t \in [0,t_*]} \|J_\nu X(t)\|^2
    \le C \sup_{[0,t^*]} \Big(|\alpha|^2  +
    \frac{|\beta|^2}{t^2}\Big)^2
    (1+\nu^4 t_*^8)\| \psi^{in} \|_{H^2}^2.
  \end{equation}
  Regarding the $L^\infty$ bound on $X$ near the poles, we restart
  from \eqref{eq_X2} to deduce
  \begin{equation*}
    \begin{aligned}
      (\partial_t - \nu \Delta) \frac{|X|^2}{2}
      & \le  \frac{\nu t}{|\ln \nu|} |X|^2 +   C |\ln \nu| \nu  \frac{|\beta|^4}{|\alpha|^2 t^4} t^3  [1- p \cdot e]^3 |\psi|^2 +  C  |\ln \nu|  \frac{|\beta|^2}{t^2} \nu t [1- p \cdot e] |\psi|^2 \\
      & \le    \frac{\nu t}{|\ln \nu|} |X|^2 + C' |\ln \nu|  \frac{|\beta|^2}{t^2}  \Big(\nu  t^3 [1-p\cdot e]^3   +   \nu t [1-p\cdot e]\Big)  \|\psi^{in}\|_{L^\infty}^2 ,
    \end{aligned}
  \end{equation*}
  where we used that $|\beta|^2/(|\alpha|^2 t^2) \lesssim 1$ for the second inequality. Mimicking the end of the proof of \cref{thm:x-infty}, we find that for all $t_* \le c \nu^{-1/2} |\ln \nu|$,
  \begin{equation} \label{X_infty_long}
    \begin{aligned}
      & \sup_{t \in [0,t_*]} |X(t,p)|^2  \le \|\nabla \psi^{in}\|_{L^\infty}  \\
      & + C  |\ln \nu|
        \Big( \sup_{[0,t_*]}  \frac{|\beta|^2}{t^2} \Big)
        \left([1-p\cdot e]^3 \nu t_*^4  + [1-p\cdot e]^2 \nu^2 t_*^5
        +  [1-p\cdot e] (\nu^3 t_*^6+\nu t_*^2)
        +  1 \right) \|\psi^{in}\|_{L^\infty}^2.
    \end{aligned}
  \end{equation}
  Taking into account that
  \begin{equation*}
    \frac{|\alpha|^2}{|\beta|^2} \sim \nu, \quad
    t^2  \lesssim \nu |\ln \nu|^2  \quad
    \text{ for } \: t \in [\nu^{-1/2}, c \nu^{-1/2} |\ln \nu|]
  \end{equation*}
  and using \eqref{X_long}-\eqref{JX_long}-\eqref{X_infty_long}, we
  can adapt the proof of \cref{sub:mixing} to obtain the second part
  of \cref{thm:vector-field-decay}.
\end{proof}

\section{Proof of mixing and enhanced diffusion for the diffusive Saintillan-Shelley model}

The good decay estimates for the semigroup \(\e^{tL_1}\) of the
previous subsection now allow us to conclude the results by studying
the Volterra equation \eqref{eq:volterra-u}.

\subsection{Proof of \cref{thm2}} \label{sub:conclusion:mixing}
We first prove \cref{thm2}.  Coming back to the Volterra equation
\eqref{eq:volterra-u}, we drop the subscript $k$ and restore the
superscript $\nu$ so that it now reads
\begin{equation} \label{Volterra_u_nu}
  u^\nu(t) + \int_0^t K^\nu(t-s)\, u^\nu(s)\, \dd s = U^\nu(t),
\end{equation}
with
\begin{equation*}
  U^\nu(t) = u[\e^{tL_{1}} \psi^{in}], \quad
  K^\nu(t) v = - u[\e^{tL_1} (\bar L_2 \cdot v)], \quad v \in \CC^3.
\end{equation*}
We remind that
\begin{equation*}
  L_1
  = - i (p \cdot k) + \nu \Delta, \qquad
  \overline{L}_2
  =  \frac{3 i \Gamma}{4\pi} (p \cdot k) \mathbb{P}_{k^\perp} p
\end{equation*}
and
\begin{equation*}
  \begin{aligned}
    u[\psi]
    :=   i  \mathbb{P}_{k^\perp}  \sigma k, \quad
    \mathbb{P}_{k^\perp} = \left(I- k \otimes k \right), \quad
    \sigma  := \eps \int_{\bbS^2} \psi(p)\, p\otimes p\, \dd p.
  \end{aligned}
\end{equation*}
Thanks to the identity
$$ (p \cdot k)   \mathbb{P}_{k^\perp} p = -  \mathbb{P}_{k^\perp}  \nabla (p \cdot k), $$
we find that
\begin{align*}
  U^\nu(t)
  & = \int_{\bbS^2} (\e^{t L_1} \psi^{in}) F \, \nabla(p \cdot k)\, \dd p,  \\
  K^\nu(t) v
  & =   \int_{\bbS^2}  (\e^{t L_1} \bar L_2 \cdot v)  F \, \nabla(p \cdot k)\, \dd p,
    \quad  F := - i \eps   \mathbb{P}_{k^\perp},
\end{align*}
Then, \cref{thm:vector-field-decay} shows that there exists a constant
$C$ independent of $\nu$ such that
\begin{equation} \label{estimates_source_kernel}
  |U^\nu(t)| \le \frac{C \ln(2+t)}{(1+t)^2}
  \|\psi^{in}\|_{H^{2+\delta}}, \qquad
  |K^\nu(t)| \le \frac{C \ln(2+t)}{(1+t)^2}, \qquad \forall t \le \nu^{-1/2}.
\end{equation}
Introduce $\tilde U^\nu := U^\nu \ind_{t \le \nu^{-1/2}}$ and
$\tilde K^\nu := K^\nu \ind_{t \le \nu^{-1/2}}$ that satisfy
\begin{equation} \label{estimates_source_kernel2}
  |\tilde U^\nu(t)|
  \le \frac{C \ln(2+t)}{(1+t)^2} \|\psi^{in}\|_{H^{2+\delta}}, \qquad
  |\tilde K^\nu(t)| \le \frac{C \ln(2+t)}{(1+t)^2}, \qquad \forall t\geq 0,
\end{equation}
and consider the corresponding modified Volterra equation
\begin{equation} \label{modified_Volterra}
  \tilde u^\nu(t)
  +  \int_0^t \tilde K^\nu(t-s)\, \tilde u^\nu(s)\, \dd s
  = \tilde U^\nu(t).
\end{equation}
Let $K^{inv}$ be the inviscid kernel considered in \cref{sec:inviscid},
that is the kernel for $\nu=0$. We have seen that $K^{inv}$ decays
like $t^{-2}$. We now claim that
\begin{equation} \label{conv_kernels}
 \lim_{\nu \rightarrow 0} \|\tilde K^\nu - K^{inv}\|_{L^1(\RR_+)}  = 0.
\end{equation}
Indeed, let $\kappa > 0$ be arbitrarily small.  We fix some large
$T > 0$ independent of $\nu$, so that
$$ \int_{T}^{+\infty} |K^{inv}| + \int_T^{+\infty} |\tilde K^\nu| \le \kappa. $$
Note that the second condition can be achieved because the constant
$C$ in \eqref{estimates_source_kernel2} is independent of $\nu$. Now,
for all $\nu$ such that $\nu^{-1/2} > T$, we write
\begin{align*}
  \|\tilde K^\nu - K^{inv}\|_{L^1(\RR_+)}
  & \le \|\tilde K^\nu - K^{inv}\|_{L^1(0,T)}
    + \int_{T}^{+\infty} \left(|\tilde K^\nu|+|K^{inv}|\right) \\
  & \le  \|K^\nu - K^{inv}\|_{L^1(0,T)} +   \kappa.
\end{align*}
Eventually, it is standard to show that, given any initial data in $L^2(\bbS^2)$,  and any finite time interval
$(0,T)$, the solution $\psi^\nu$ of the advection-diffusion equation
$$\partial_t \psi^\nu + i (p \cdot k)   \psi^\nu - \nu \Delta \psi^\nu = 0$$
converges in $L^\infty(0,T, L^2(\bbS^2))$ to the solution $\psi^{inv}$
of $\partial_t \psi^{inv} + i (p \cdot k) \psi^{inv} = 0$ (with the
same initial data). With the choice of initial data
$\psi_{in} =\bar L_2 \cdot v$, it follows that
$\lim_{\nu \rightarrow 0}\|K^\nu - K^{inv}\|_{L^1(0,T)} = 0$, and the
claim \eqref{conv_kernels} follows.

From there, we can conclude by the stability of solutions of Volterra
equations for $\nu$ small enough equation \eqref{modified_Volterra}
has a solution, that decays like $O(\ln(2+t)/t^2)$. More precisely,
one can construct perturbatively the resolvent $\tilde R^\nu$,
satisfying
\begin{equation*}
  \tilde R^\nu + \tilde K^\nu \star \tilde R^\nu = \tilde K^\nu,
\end{equation*}
see \cref{thm3.9_Volterra}, which even applies to non-convolution
kernels. In particular, the stability condition
$\det (I + \cL \tilde K^\nu(z)) \neq 0$ in $\{\Re z \ge 0\}$ is
satisfied, and $\|\tilde R^\nu\|_{L^1} \le C$. Using the estimates in
\eqref{estimates_source_kernel2} and \cref{thm_decay_Volterra} (see
also \cref{rem:FaouRousset}), we deduce that
\begin{equation*}
  | \tilde u^\nu(t)|
  \le \frac{C \ln(2+t)}{(1+t)^2} \|\psi^{in}\|_{H^{2+\delta}},
  \quad t \ge 0.
\end{equation*}
But, as $\tilde K^\nu = K^\nu$ on $[0, \nu^{-1/2})$, by uniqueness of
the solutions of the Volterra equation on a finite time interval, one
has $\tilde u^\nu = u^\nu$ on $[0, \nu^{-1/2})$, and the first
inequality in \cref{thm2} follows. From there, the second one can be
obtained exactly as in \cref{subsec:control_psi}, by relating $\psi$
to $u$ through Duhamel's formula. This concludes the proof of the
theorem.

\subsection{Proof of \cref{thm3}} \label{sub:conclusion:enhanced}

Here we use the additional  decay estimates from the enhanced
dissipation of \cref{prop:hypocoercf} and the extended time-frame of
\cref{thm:vector-field-decay}.  Using the second inequality in
\cref{thm:vector-field-decay}, we have
\begin{equation} \label{ineqK1}
  |K^\nu(t)| \lesssim  \nu |\ln \nu|^M, \quad
  |U^\nu(t)|   \lesssim  \nu |\ln \nu|^M
  \|\psi^{in}\|_{H^{2+\delta}}, \quad
  \forall t  \in [\nu^{-1/2}, c \nu^{-1/2} |\ln \nu |].
\end{equation}
Moreover, by \cref{prop:hypocoercf}, we also have for $\nu$ small
enough
\begin{equation} \label{ineqK2}
  |K^\nu(t)| \lesssim \e^{-\eps_0 \nu^{1/2} t}, \quad
  |U^\nu(t)| \lesssim  \e^{-\eps_0 \nu^{1/2} t} \|\psi^{in}\|_{H^{2+\delta}}.
\end{equation}
We first prove that $\|K^\nu - K\|_{L^1(\RR_+)} \rightarrow 0$ as
$\nu \rightarrow 0$. We fix $c$ such that $c \eps_0 > \frac{1}{2}$.
We decompose
\begin{align*}
  \|K^\nu - K\|_{L^1} \le \int_0^T |K^\nu - K|    +\int_{T}^{\infty} |K|  + \int_T^{\nu^{-1/2}} |K^\nu|   +  \int_{\nu^{-1/2}}^{c \nu^{-1/2} |\ln \nu|} |K^\nu| + \int_{c \nu^{-1/2} |\ln \nu |}^{\infty} |K^\nu|.
\end{align*}
The first three terms can be treated as in
\cref{sub:conclusion:mixing}: for any $\kappa > 0$ there exists $T$ large
enough and $\nu$ small enough so that
\begin{equation*}
  \int_0^T |K^\nu - K|
  + \int_{T}^{\infty} |K|  + \int_T^{\nu^{-1/2}} |K^\nu|
  \le  \kappa.
\end{equation*}
For the fourth term, we use \eqref{ineqK1}, that yields
\begin{equation*}
  \int_{\nu^{-1/2}}^{c \nu^{-1/2} |\ln \nu|} |K^\nu|
  \lesssim  \nu^{1/2} |\ln \nu|^{M+1} \le \kappa,
\end{equation*}
for $\nu$ small enough.  Finally,
\begin{equation*}
  \int_{c \nu^{-1/2} |\ln \nu |}^{\infty} |K^\nu|
  \lesssim
  \int_{c \nu^{-1/2} |\ln \nu |}^{\infty} \e^{-\eps_0 \nu^{1/2} t}
  \lesssim \nu^{c\eps_0 - \frac{1}{2}}  \le \kappa,
\end{equation*}
for $\nu$ small enough. Hence,
$\|K^\nu - K\|_{L^1(\RR_+)} \rightarrow 0$ as $\nu \rightarrow
0$. Arguing as in \cref{sub:conclusion:mixing}, we deduce that the
spectral condition $\det (I + \cL K^\nu(z)) \neq 0$ in
$\{\Re z \ge 0\}$ is satisfied, and that the resolvent $R^\nu$
associated to $K^\nu$ satisfies $\|R^\nu\|_{L^1} \le C$.

We now would like to establish the bound
\begin{equation} \label{ineq_unu1}
  |u^\nu(t)| \lesssim \e^{- \eta_0 \nu^{1/2}t} \|\psi^{in}\|_{H^{2+\delta}}, \qquad \eta_0 := \frac{\eps_0}{4}.
\end{equation}
We come back to the proof of \cref{thm_decay_Volterra}, replacing the
weight $(1+\eps t)^\alpha$ by $\e^{-\eta_0 \nu^{1/2}t}$. The
non-convolution kernel
$k(t,s) = \frac{(1+\eps t^\alpha)}{(1+\eps s^\alpha)} K(t-s) 1_{t-s >
  0}$ is now replaced by
\begin{align*}
  k^\nu(t,s) \:
  & :=  \:  \e^{\eta_0 \nu^{1/2} (t-s)} K^\nu(t-s) 1_{t-s > 0} \\
  &  =   K_\nu(t-s) 1_{t-s > 0}  +  \big( \e^{\eta_0 \nu^{1/2} (t-s)}
    -1 \big) K^\nu(t-s) 1_{t-s > 0}
    =: k^\nu_1(t,s) + k^\nu_2(t,s).
\end{align*}
It is easily seen that $k^\nu$ is a Volterra kernel of type
$L^\infty$. Moreover, by the previous considerations, for $\nu$ small
enough, $k_1^\nu$ has the resolvent
$r_1(t,s) := R^\nu(t-s) 1_{t-s > 0}$, which satisfies
$\|r^\nu_1\|_{\infty,\RR_+} \le \|R^\nu\|_{L^1(\RR_+)} \le C$ for a
constant $C$ independent of $\nu$. By using \cref{thm3.9_Volterra}, it
is enough to show that $\|k_2^\nu\|_{\infty,\RR_+}$ can be made
arbitrarily small for $\nu$ small enough. We fix
$c= \frac{4}{3\eps_0} = \frac{1}{3\eta_0}$, and decompose for some
$\delta > 0$ to be fixed
\begin{align*}
  \|k_2^\nu\|_{\infty,\RR_+}
  & \le \int_{0}^{\infty}  \Big|  \big( \e^{\eta_0 \nu^{1/2} s} -1 \big) K_1^\nu(s)  \Big|  \dd s \\
  & \le  \Big( \int_{0}^{\delta \nu^{-1/2}}
    + \int_{\delta \nu^{-1/2}}^{c |\ln \nu | \nu^{-1/2}}
    + \int_{c |\ln \nu | \nu^{-1/2}}^{\infty}  \Big)
    \Big|  \big( \e^{\eta_0 \nu^{1/2} s} -1 \big) K_1^\nu(s)
    \Big|\, \dd s
    =:  I_1 + I_2 + I_3.
\end{align*}
Let $\kappa > 0$. Using the uniform bound $K^\nu = O(\ln(2+t) t^{-2})$
on $[0,\nu^{-1/2}]$, we find
$$ I_1 \lesssim (\e^{\eta_0 \delta} - 1) \le \kappa \quad \text{for $\delta$ small enough}. $$
This $\delta$ being fixed, using the bound
$|K^\nu| \lesssim \nu |\ln \nu|^M$ on
$[\delta \nu^{-1/2}, c \nu^{-1/2} |\ln \nu | ]$, which follows from
\eqref{estimates_source_kernel} and \eqref{ineqK1}, we find
\begin{equation*}
  I_2 \lesssim \nu^{-\eta_0 c}  \nu^{1/2} |\ln \nu|^{M+1}
  \lesssim \nu^{1/2 - 1/3} \le \kappa,
\end{equation*}
for $\nu$ small enough. Eventually, using \eqref{ineqK2}, we get
\begin{equation*}
  I_3 \lesssim \nu^{(\eps_0 - \eta_0) c - 1/2} \lesssim \nu^{1/2} \le \kappa,
\end{equation*}
for $\nu$ small enough. Hence, for $\nu$ small enough, $k^\nu$ has a
resolvent $r^\nu$ that satisfies $\|r^\nu\|_{\infty,\RR_+} \le
C$. Also, by the bound \eqref{ineqK2}, we know that
$\tilde U^\nu = \e^{\eta_0 \nu^{1/2}t} U^\nu$ is bounded by
$C \|\psi^{in}\|_{H^{2+\delta}}$, and finally
\begin{equation*}
  \tilde u^\nu = \tilde U^\nu - \int_{\RR_+} r^\nu(t,s) \tilde
  U^\nu(s) \dd s,
\end{equation*}
is bounded as well, which yields that
$u^\nu = \e^{-\eta_0 \nu^{1/2}t} \tilde u^\nu$ satisfies
\eqref{ineq_unu1}.

For the intermediate time regime, we fix $c$ such that $\eta_0 c =
2$. Combining \eqref{estimates_source_kernel} and \eqref{ineqK1}, we
find that
\begin{equation*}
  |K_\nu(t)| \lesssim \frac{\ln(2+t)^{M+2}}{(1+t)^2},
  \quad \forall t \in [0, c \nu^{-1/2} |\ln \nu|].
\end{equation*}
Arguing exactly as in \cref{sub:conclusion:mixing}, we deduce:
\begin{equation*}
  |u^\nu(t)| \lesssim
  \frac{\ln(2+t)^{M+2}}{(1+t)^2} \|\psi^{in}\|_{H^{2+\delta}},
  \quad \forall t \in [0, c \nu^{-1/2}| \ln \nu|],
\end{equation*}
which in particular gives for small enough $\nu$:
\begin{equation} \label{ineq_unu2}
  |u^\nu(t)| \lesssim  \nu |\ln(\nu)|^{M+2}
  \|\psi^{in}\|_{H^{2+\delta}},
  \quad \forall t \in [\nu^{-1/2}, c \nu^{-1/2}| \ln \nu|],
\end{equation}
Moreover, with our choice of $c$, for $\nu$ small enough and
$t \ge c \nu^{-1/2} |\ln \nu|$,
\begin{equation*}
  \e^{-\eta_0 \nu^{1/2}t}  \lesssim  \nu^2 \ll   \nu |\ln\nu|^{M+2}
\end{equation*}
so that \eqref{ineq_unu1} and \eqref{ineq_unu2} imply the first
inequality in \cref{thm3} (replacing $M$ by $M+2$).

To obtain the control of $\psi = \psi^\nu$, we use the expression
\begin{equation*}
  \psi^\nu(t)
  = \e^{t L_1} \psi^{in}
  +  \int_0^t  \left( \e^{(t-s)L_1} \bar L_2\right) \cdot u^\nu(s) \, \dd s.
\end{equation*}
By \cref{prop:hypocoercf} and \eqref{ineq_unu1}, we get
\begin{equation*}
\begin{aligned}
  \|\psi^\nu(t)\|
 & \lesssim \e^{-\eps_0 \nu^{1/2}t} \|\psi_{in}\|
  + \int_0^t \e^{-\eps_0 \nu^{1/2}(t-s)} \e^{-\eta_0 \nu^{1/2} s}\,
  \dd s  \|\psi_{in}\|_{H^{2+\delta}}  \\
&  \lesssim \e^{-\eps_0 \nu^{1/2}t}  \|\psi_{in}\|
  + \e^{-\min(\eps_0,\eta_0) \nu^{1/2} t}
  \int_0^t \dd s \|\psi_{in}\|_{H^{2+\delta}} ,
  \end{aligned}
\end{equation*}
which implies the last bound in \cref{thm3}.

\begin{remark}[Stability of pullers]
  The spectral stability condition
  \begin{equation*}
    \det (I + \cL \tilde K^\nu(z)) \neq 0, \qquad \text{ in}\quad \{\Re z \ge 0\}
  \end{equation*}
  was shown above to be satisfied for small $\nu$, through
  perturbation of the inviscid condition. Actually, in the case of
  pullers, there is a more straightforward way, valid for all $\nu$
  and directly for the original kernel $K^\nu$. Indeed, one can simply
  notice that
  \begin{equation*}
    \mathcal{L} K^\nu(z) v \cdot \conj{v}
    =  \frac{3\Gamma}{4\pi}\int (z-L_1)^{-1} \phi\, \conj{\phi},
  \end{equation*}
  where \(\phi = (p\cdot k) \mathbb{P}_{k^\perp} p \cdot v\). In terms
  of \(F = (z-L_1)^{-1} \phi\) we thus find
  \begin{equation*}
    \mathcal{L} K^\nu(z)  v \cdot \conj{v}
    =  \frac{3\Gamma}{4\pi}\ \l (z-L_1) F, F \r
    = \frac{3\Gamma}{4\pi}\
    \Big( \Re z \| F \|^2 + \nu \| \nabla F \|^2 \Big)
    \ge 0 \quad \text{for } \: \Re z \ge 0.
  \end{equation*}
  This implies in particular that $\det(I + \cL K^\nu)$ cannot vanish
  in the unstable half plane.
\end{remark}

\section*{Acknowledgements}
The authors would like to thank the Isaac Newton Institute for
Mathematical Sciences, Cambridge, for support and hospitality during
the programmes \emph{Mathematical aspects of turbulence} and
\emph{Frontiers in kinetic theory} where part of the work on this
paper was undertaken. DGV is also grateful to the Mathematics
Department of Imperial College for its hospitality and its support
through an ICL-CNRS fellowship, which made possible the collaboration.
This work was supported by EPSRC grant no EP/R014604/1 and a grant
from the Simons Foundations.  MCZ acknowledges funding from the Royal
Society through a University Research Fellowship (URF\textbackslash
R1\textbackslash 191492). HD and DGV acknowledge the grant ANR-18-
CE40-0027 of the French National Research Agency (ANR).

\appendix
\section{Reminders on covariant derivatives}
\label{app:covariant}

For the computations on the sphere, it is natural to rely on covariant
derivatives from Riemannian geometry. Classical references are
\cite{Petersen16,Lee18}, we just provide here a quick reminder.  We
start by introducing the natural metric \(g\) on the sphere induced by
the Euclidean scalar product on $\RR^3$. In spherical coordinates
\((\theta,\varphi)\) it is explicitly given by
\begin{equation*}
  g_{ij} =
  \begin{pmatrix}
    1 & 0 \\
    0 & \sin^2 \theta
  \end{pmatrix}
  \quad
  \Rightarrow
  \quad
  g^{ij} =
  \begin{pmatrix}
    1 & 0 \\
    0 & \sin^{-2} \theta
  \end{pmatrix}.
\end{equation*}
Given this metric and the associated Levi-Civita connection, we
introduce the covariant derivative of a tensor, denoted by
\(\nabla\). We remind that the covariant derivative of a vector $X$,
resp. of a covector \(\omega\), is the $(1,1)$-tensor, resp. the
$(0,2)$ tensor, defined in a coordinate basis by
\begin{equation*}
  \nabla_i X^j
  = \partial_i X^{j} + \Gamma^j_{ik} X^k,
\end{equation*}
resp.
\begin{equation*}
  \nabla_i \omega_j
  = \partial_i \omega_{j} - \Gamma^k_{ij} \omega_k,
\end{equation*}
containing the Christoffel symbols \(\Gamma\) defined as
\begin{equation*}
  \Gamma^k_{ij} = \frac 12 g^{kl}
  (\partial_j g_{il} + \partial_i g_{jl} - \partial_l g_{ij}).
\end{equation*}
For general tensors, the covariant derivative is defined recursively
through the formula
$$ \nabla (T \otimes T') = \nabla T \otimes T' + T \otimes \nabla T' $$
The basic properties of the covariant derivatives are that they
commute with the metric $g$ and that they satisfy the Leibniz rule.

Thanks to this covariant derivative, we then define the connection
Laplacian of a tensor as
$$\Delta = \tr(\nabla^2) = g^{ij} \nabla_i \nabla_j.$$

Higher-order derivatives do not commute due to the curvature. The
Ricci identity captures the defect by the Riemann curvature tensor
\(R\). In a coordinate basis, it takes for a covector \(\omega\) the
form
\begin{equation*}
  [\nabla_a,\nabla_b] \omega_c
  = - {R^{d}}_{cab} \omega_d.
\end{equation*}

By the definition of the connection Laplacian, we thus find for a
scalar function \(f\) that
\begin{equation*}
  [\Delta,\nabla] f = \ricci \nabla f
\end{equation*}
where $\ricci$ denotes the Ricci curvature, obtained by contracting
the first and third argument in the Riemann curvature tensor. On a
unit-sphere the Ricci curvature just equals the metric and thus we
find
\begin{equation*}
  [\Delta,\nabla] f = \nabla f
\end{equation*}
which yields a good sign for our estimates in
\cref{sec:diffusive}. The good sign simplifies the algebra even though
the estimates would equally work for a bounded curvature.

Acting on a general tensor \(X\), we find a similar expression with a
successive application of the Riemann curvature tensor. As the Riemann
curvature tensor is bounded on a sphere, we just note that
\begin{equation*}
  [\Delta,\nabla] X = \mathcal{R} \nabla X
\end{equation*}
for a bounded tensor \(\mathcal{R}\).

In the enhanced dissipation estimates, we use the commutator between
the Laplacian and the tensor $(p \cdot e) X$, for a fixed vector
\(e\). It is provided by
\begin{lemma}
  \label{thm:laplace-mixed-term}
  Let \(X\) be a \((0,N)\)-tensor. Then it holds that
  \begin{equation*}
    \Delta(\nabla(p\cdot e) \otimes X)
    = - \nabla(p\cdot e) \otimes X
    - 2 (p\cdot e) \nabla X
    + \nabla(p\cdot e) \otimes \Delta X.
  \end{equation*}
\end{lemma}
\begin{proof}
  We compute the expression in the spherical coordinates
  \((\theta,\varphi)\) where we take \(e\) to be the north pole. Using
  the Leibniz rule, we find that
  \begin{equation*}
    \Delta(\nabla_{i_0}(p\cdot e) X_{i_1,\dots,i_n})
    = (\Delta \nabla_{i_0}(p\cdot e)) X_{i_1,\dots,i_n})
    + 2 g^{mn} (\nabla_m \nabla_{i_0} (p\cdot e))
    \nabla_n X_{i_1,\dots,i_n}
    + \nabla_{i_0}(p\cdot e) \Delta X_{i_1,\dots,i_n}.
  \end{equation*}
  We can explicitly compute that
  \begin{equation*}
    \Delta (\nabla(p\cdot e))
    = \nabla    \Delta (p\cdot e) + \ricci  \nabla (p \cdot e)
    =  \nabla  \Delta  \cos \theta + \nabla \cos \theta
    = \sin \theta\, \dd\theta
    = - \nabla (p \cdot e).
  \end{equation*}
  For the mixed term, we find
  \begin{equation*}
    \nabla_m \nabla_{i_0} \cos \theta
    = \partial_m \partial_{i_0} \cos \theta
    - \Gamma^k_{mi_0} \partial_k \cos \theta.
  \end{equation*}
  In the spherical coordinates, the only non-zero Christoffel symbols
  are
  \begin{equation*}
    \Gamma^{\theta}_{\varphi\varphi} = - \cos\theta \sin\theta
    \quad \text{ and } \quad
   \Gamma^{\varphi}_{\varphi\theta} =
    \Gamma^{\varphi}_{\theta\varphi} = \frac{\cos \theta}{\sin\theta}.
  \end{equation*}
  This then yields that
  \begin{equation*}
    g^{mn} (\nabla_m \nabla_{i_0} (p\cdot e))
    \nabla_n X_{i_1,\dots,i_n}
    = (p\cdot e) \nabla_{i_0}X_{i_1,\dots,i_n}.
  \end{equation*}
  Hence we have arrived at the claimed expression.
\end{proof}

\subsection*{Conflict of interest}
The authors declare that they have no conflict of interest.
\bibliographystyle{abbrv}
\bibliography{lit}
\end{document}